%% file: AFC_Source_Terms.tex
\documentclass[a4paper]{article}

\input{preamble.tex}

\input macros.tex

\theoremstyle{remark}
\newtheorem{remark}{Remark}
\theoremstyle{theorem}
\newtheorem{lemma}{Lemma}
\newtheorem{theorem}{Theorem}
\newtheorem{corollary}{Corollary}

  \title{Well-balanced convex limiting for finite
    element discretizations of
    steady convection-diffusion-reaction equations}

\author{Petr Knobloch\footnote{Department of Numerical Mathematics, Faculty of Mathematics
and Physics, Charles University, Sokolovsk\'a 83, Praha 8, 18675,
Czech Republic, \texttt{knobloch@karlin.mff.cuni.cz}},  Dmitri Kuzmin\footnote{Institute of Applied Mathematics (LS III), TU Dortmund University,\\ Vogelpothsweg 87,
  D-44227 Dortmund, Germany, \texttt{kuzmin@math.uni-dortmund.de}},  Abhinav Jha\footnote{Institute of Applied Analysis and Numerical Simulation,
University of Stuttgart,  Pfaffenwaldring 57, 70569 Stuttgart, Germany, \texttt{abhinav.jha@ians.uni-stuttgart.de}}}

\date{}

\begin{document}
\maketitle
\begin{abstract}
  We address the numerical treatment of source terms in algebraic flux correction schemes for
  steady convection-diffusion-reaction (CDR) equations. The proposed algorithm constrains a continuous
  piecewise-linear finite element approximation using a monolithic convex limiting (MCL) strategy.
  Failure to discretize the convective derivatives  and source terms in a compatible manner produces
  spurious ripples, e.g., in regions where the coefficients of the continuous problem
  are constant and the exact solution is linear. We cure this deficiency by incorporating
  source term components into the fluxes and intermediate states of the MCL procedure.
  The design of our new limiter is motivated by the desire to preserve simple steady-state
  equilibria exactly, as in well-balanced schemes for the shallow water equations. 
  The results of our numerical experiments for two-dimensional CDR problems
  illustrate potential benefits of well-balanced flux limiting in the scalar case.
\end{abstract}

\section{Introduction}\label{sec:intro}

Many modern numerical schemes for conservation laws are equipped with flux or slope limiters
that ensure the validity of discrete maximum principles. A comprehensive review of such
algorithms and of the underlying theory can be found, e.g., in \cite{kuzmin2023}. Matters
become more complicated in the case of inhomogeneous balance laws, especially if strong consistency with
some
steady-state solutions is desired. Discretizations that provide such consistency
are called \emph{well balanced} in the literature
\cite{audusse2015,fjordholm2011,noelle2006}. For example, a well-balanced
scheme for the system of shallow water equations (SWEs) should preserve at least lake-at-rest equilibria
(zero velocity, constant free surface elevation). In general, sources/sinks should be
discretized in a manner compatible with the numerical treatment of flux terms \cite{leveque1998}.
In the one-dimensional case, proper balancing can often be achieved by discretizing a
`homogeneous form’ of the balance law \cite{donat2011,gascon2001,sweby1989}. The design of
well-balanced
schemes for multidimensional problems is usually more difficult, especially if
the source term does not admit a natural representation as the gradient of a scalar
potential or divergence of a vector field.

A well-balanced and positivity-preserving finite element scheme for the inhomogeneous SWE system was
developed by Hajduk \cite{hajduk2022diss} using the framework of algebraic flux correction. The
monolithic convex limiting (MCL) algorithm presented in \cite{hajduk2022diss,hajduk2022} 
incorporates discretized bathymetry gradients into the numerical fluxes and intermediate states
of the spatial semi-discretization. In the present paper, we show that the source term of a
scalar convection-diffusion-reaction problem can be treated similarly. In particular, we
define numerical fluxes that ensure consistency of the well-balanced MCL approximation
with a linear steady state. Using a convex decomposition into intermediate states, we enforce
positivity preservation, as well as local {and global discrete} maximum principles. 

In Section \ref{sec:std_galerkin}, we discretize a model problem using the
standard continuous Galerkin {finite element} method. The
algorithm presented in Section  \ref{subsec:convective} stabilizes the convective part using the MCL methodology
for hyperbolic conservation laws \cite{Ku20,kuzmin2023}. 
The discretization of source terms is left unchanged
in this version. Our well-balanced generalization is derived in Section
\ref{sec:well_balanced}, {analyzed in Section
\ref{sec:analysis},} and tested numerically
in Section \ref{sec:numres}. The numerical examples with locally linear exact solutions show that improper
treatment of source terms may cause a flux-corrected finite element method to produce spurious ripples. The proposed approach provides an effective
remedy to this problem. Section \ref{sec:summary} closes the
 paper with a summary and discussion of the main findings.

\section{Model problem and Galerkin discretization}\label{sec:std_galerkin}

In computational fluid dynamics, steady convection-diffusion-reaction (CDR) equations are often used to
simulate distributions of scalar quantities of interest, such as temperature, energy, or concentration
of chemical species. Let $d\in \lbrace 1, 2,3\rbrace$ denote the number of space dimensions.
Choosing a domain $\Omega\subset \mathbb{R}^d$
with Lipschitz boundary $\Gamma = \partial \Omega$,
we consider the Dirichlet problem
\begin{subequations}\label{bvp}
  \begin{alignat}{3}\label{eq:cdr_eqn}
-\varepsilon \Delta u +  \mathbf{v}\cdot \nabla u+  c\,u  & =  f  &&\quad \mathrm{in}\ \ \Omega,\\
u & =  u_D \ &&\quad \mathrm{on}\ \ \Gamma_D,\label{eq:cdr_dbc}
\end{alignat}
\end{subequations}
where $u=u(\mathbf{x})$ is the unknown variable, $\varepsilon\ge 0$ is a constant diffusion coefficient,
$\mathbf{v}=\mathbf{v}(\mathbf{x})$ is a given velocity field, $c=c(\mathbf{x})$ is a nonnegative reaction
rate, and $f=f(\mathbf{x})$ is a general source term depending on the vector $\mathbf{x}
=(x_1,\ldots,x_d)^\top$ of space coordinates. In the case $\varepsilon>0$, the  Dirichlet  boundary
data $u_D$ is prescribed on $\Gamma_D=\Gamma$. In the case $\varepsilon=0$, equation
\eqref{eq:cdr_eqn} becomes hyperbolic and, therefore, condition \eqref{eq:cdr_dbc} is imposed
only on the inflow boundary $\Gamma_D\subseteq\Gamma$.

We are particularly interested in the case of dominating convection. Thus we assume that
$\|\mathbf{v}\|_{(L^{\infty}(\Omega))^d} \gg
\frac{\varepsilon}{L}$, where $L$ is the characteristic length of the problem. Because of this assumption,
an exact solution to \eqref{bvp} may exhibit interior and/or boundary layers, in
which the gradients are steep and standard finite element methods may violate
discrete maximum principles \cite{RST08}.
\smallskip

Let $\mathcal{T}_h$ be a conforming simplex mesh such that $\bigcup_{K\in \mathcal{T}_h}K = \overline{\Omega}$. The vertices of $\mathcal{T}_h$ are denoted by $\mathbf{x}_j,\ j\in\{1,\ldots,N_h\}$ and the maximum diameter of mesh cells $K\in \mathcal{T}_h$ by $h>0$. Restricting our discussion to linear finite elements in this paper, we express numerical approximations
\beq\label{uhdef}
u_h=\sum_{j=1}^{N_h}u_j\varphi_j
\eeq
in terms of Lagrange basis functions $\varphi_j\in C(\bar\Omega),\
j\in \{1,\ldots,N_h\}$ such that $\varphi_j|_K\in\mathbb{P}_1(K)\ \forall K\in\mathcal T_h$ and $\varphi_j(\mathbf{x}_i)=\delta_{ij}$ for $i\in \{1,\ldots,N_h\}
$. The corresponding finite element space is denoted by $V_h$.

We assume that the Dirichlet boundary nodes
are numbered using indices $M_h+1,\ldots,N_h$.
 Substituting
\eqref{uhdef} into the discretized weak form
$$
\int_\Omega\left(\varepsilon \nabla w_h\cdot\nabla u_h+ w_h[  \mathbf{v}\cdot \nabla u_h+  c\,u_h]\right)\dx
=\int_\Omega w_hf\dx
$$
of \eqref{eq:cdr_eqn}
and using test functions $w_h\in\{\varphi_1,\dots,\varphi_{M_h}\}$, we obtain a linear system for the unknown nodal values:
\begin{subequations}\label{gal}
  \begin{alignat}{3}\label{eq:gal_eq}
  \sum_{j=1}^{N_h}(a_{ij}^D+a_{ij}^C+a_{ij}^R)u_j&=b_i,&& i=1,\ldots,M_h,\\
  u_i&=u_D(\mathbf{x}_i),\qquad&& i=M_h+1,\ldots,N_h.\label{eq:gal_bc}
\end{alignat}
\end{subequations}
The coefficients of the involved matrices and vectors are given by
\begin{align*}
   a_{ij}^D&=\varepsilon\int_{\Omega}\nabla \varphi_i\cdot \nabla \varphi_j\dx, 
   \qquad
   a_{ij}^C=\int_{\Omega} \varphi_i\mathbf{v}\cdot\nabla \varphi_j\dx,\\
   a_{ij}^R&=\int_{\Omega}c\, \varphi_i\varphi_j\dx,\qquad
   b_i=\int_{\Omega}\varphi_if\dx.
\end{align*}
The matrices 
$A^D=(a_{ij}^D)_{i,j=1}^{N_h}, \ A^C=(a_{ij}^C)_{i,j=1}^{N_h}$, and 
$A^R=(a_{ij}^R)_{i,j=1}^{N_h}$
result from the discretization of the diffusive, convective, and reactive
terms, respectively. The contribution of the right-hand side $f$ is
represented by the vector $b=(b_i)_{i=1}^{M_h}$ of discretized source terms.
\smallskip

In the next section, we stabilize the discrete convection operator $A^C$ using an algebraic
flux correction scheme that satisfies a local discrete maximum principle (DMP)
in the case $\varepsilon=0$. To ensure the DMP property of the
discrete diffusion operator $A^D$ for $\varepsilon>0$,
  we assume that $a_{ij}^D\le 0$ for $j\ne i\in\{1,\ldots,N_h\}$. This requirement is
 met for simplex meshes of weakly acute type \cite{BJK23}.

\section{Convex limiting for convective terms}\label{subsec:convective}

Let $\mathcal N_i$ denote the set of indices $j$ such that the basis functions $\varphi_i$
and $\varphi_j$ have overlapping supports. 
For our purposes, it is worthwhile to write the $i$th equation of
\eqref{eq:gal_eq} in the form
\beq\label{galR}
a_i^Ru_i+\sum_{j\in\Ni\backslash\{i\}}(a_{ij}^D+a_{ij}^C+a_{ij}^R)(u_j-u_i)=b_i,
\eeq
where
$a_i^R=\sum_{j\in\Ni}a_{ij}^R$
is a diagonal entry of the `lumped' reactive mass matrix $\tilde A_R=(a_i^R\delta_{ij})_{i,j=1}^{N_h}$. 

Introducing an artificial diffusion (graph Laplacian)
operator $D=(d_{ij})_{i,j=1}^{N_h}$ with entries\footnote{We use
  a small constant $\delta>0$ 
to prevent
  division by zero without considering special cases.} 
$$
d_{ij}=\begin{cases}
\max\lbrace |a_{ij}^C|,\delta h^{d-1},|a_{ji}^C|\rbrace & \mbox{if}\
{j\in\Nis,}\\
{0}& {\mbox{if}\ j\not\in\Ni,}\\
-\sum_{k\in\Nis}d_{ik} & \mbox{if}\ j=i,
\end{cases}
$$
we define auxiliary \emph{bar states} $\bar u_{ij}$ and numerical fluxes
$f_{ij}$ for $j\in\Ni\backslash\{i\}$ as follows:
\begin{equation}\label{eq:ubar}
\bar u_{ij}=\frac{u_j+u_i}2-\frac{a_{ij}^C(u_j-u_i)}{2d_{ij}},\qquad
f_{ij}=(d_{ij}+a_{ij}^R)(u_i-u_j).
\end{equation}

It is easy to verify that the Galerkin discretization \eqref{galR} is equivalent to
\beq\label{gal:bar}
a_i^Ru_i-\sum_{j\in\Nis}[2d_{ij}(\bar u_{ij}-u_i)+f_{ij} -a_{ij}^D(u_j-u_i)]=b_i.
\eeq
The monolithic convex limiting (MCL) algorithm developed in \cite{Ku20} for
homogeneous hyperbolic problems replaces the
target flux $f_{ij}=-f_{ji}$ with an approximation $f_{ij}^*=-f_{ji}^*$ 
such that 
$$
\min_{j\in\Ni}u_j=:u_i^{\min}\le \bar u_{ij}^*:=\bar u_{ij}
+\frac{f_{ij}^*}{2d_{ij}}\le u_i^{\max}:=\max_{j\in\Ni}u_j.
$$
These inequality constraints imply the validity of local DMPs and are satisfied for
\begin{equation}\label{MClimiter}
f^*_{ij} =
	\begin{cases} 
	\min\left\lbrace f_{ij}, \min \left\lbrace 2d_{ij} \left(u_i^{\max}-\bar{u}_{ij}\right), 2d_{ij} \left(\bar{u}_{ji}-u_j^{\min}\right) \right\rbrace \right\rbrace &\text{if } f_{ij}>0, \\
	0 &\text{if } f_{ij}=0, \\
	\max\left\lbrace f_{ij}, \max \left\lbrace 2d_{ij} \left(u_i^{\min}-\bar{u}_{ij}\right), 2d_{ij} \left(\bar{u}_{ji}-u_j^{\max}\right)\right\rbrace \right\rbrace &\text{if } f_{ij}<0.
	\end{cases}
\end{equation}
To avoid division by $d_{ij}$ in the formula for $\bar u_{ij}$,
the products  $2d_{ij}\bar u_{ij}$ are calculated directly in
practical implementations of \eqref{MClimiter}. We refer the reader to
\cite{Ku20,kuzmin2023} for further explanations and proofs of local DMPs that are valid
in the limit of pure convection (i.e., for $\varepsilon=0,\ c\equiv 0$, $f\equiv 0$).

\section{Well-balanced convex limiting}\label{sec:well_balanced}

As mentioned in the introduction, a well-designed numerical scheme should be consistent with simple
steady-state equilibria. An exact solution of the CDR equation \eqref{eq:cdr_eqn} with
\beq\label{constcoeff}
\varepsilon\ge 0,\quad\mathbf{v}\equiv\hat{\mathbf{v}},\quad
    c\equiv 0,\quad  f\equiv  \hat f,
 \eeq
    where $\hat{\mathbf{v}}\in \R^d\backslash\{\mathbf 0\}$ and
    $\hat f\in\R\backslash\{0\}$ are constant, is given by
 \beq\label{uhatdef}
\hat u(\mathbf{x})=\hat f\ \frac{\mathbf{x}\cdot\hat{\mathbf{v}}}{|\hat{\mathbf{v}}|^2}.
\eeq
We denote by $|\cdot|$ the Euclidean norm of vectors in $\R^d$.
By the linearity of $\hat u$, we have
$$
 \hat{\mathbf{v}}\cdot\nabla \hat u-\varepsilon\Delta \hat u=
 \hat{\mathbf{v}}\cdot\nabla\hat u=\hat f.
 $$
The equilibrium state $\hat u$  is preserved exactly by the standard Galerkin discretization because the
linear function $\hat u(\mathbf{x})$ belongs to the space $V_h$. However, this desirable property
may be lost if an algebraic stabilization of convective terms is not balanced by an appropriate
modification of $b_i$.
\smallskip

To derive a well-balanced MCL scheme for problem \eqref{bvp} with
velocity $\mathbf{v}=\mathbf{v}(\mathbf{x})$ such that
$$|\mathbf{v}(\mathbf{x}_i)|+|\mathbf{v}(\mathbf{x}_j)|>0,
\qquad i=1,\dots,N_h,\quad j\in\Nis,
$$
 we introduce the 
\emph{balancing fluxes}
\beq\label{balflux}
P_{ij}=
\frac12\frac{s_i+s_j}2\frac{(\mathbf{x}_i-\mathbf{x}_j)\cdot
(\mathbf{v}(\mathbf{x}_i)+\mathbf{v}(\mathbf{x}_j))
}{2\,
\left(\max\{|\mathbf{v}(\mathbf{x}_i)|,|\mathbf{v}(\mathbf{x}_j)|\}\right)^2},\qquad
i=1,\dots,N_h,\quad j\in\Nis.
\eeq
In this formula, $s_i:=f(\mathbf{x}_i)-c(\mathbf{x}_i)u_i$ 
is the net source term of the CDR equation \eqref{eq:cdr_eqn} 
evaluated at the vertex $\mathbf{x}_i$. 
We replace the bar state $\bar u_{ij}$
 of representation \eqref{gal:bar}
with (cf. \cite{hajduk2022})
$$
\bar u_{ij}^s=\bar u_{ij}+\alpha_{ij}P_{ij}
+\frac{b_i}{a_i^C},\qquad
{i=1,\dots,M_h,\quad j\in\Nis,}
$$
where $a_i^C=\sum_{j\in\Nis}2d_{ij}$ and
$\alpha_{ij}=\alpha_{ji}$ is a correction factor to be defined below. By definition of $d_{ij}$, the
 coefficient $a_i^C$ is strictly positive. 
 Hence, no
 division by zero can occur. 
 \smallskip
 

The standard Galerkin discretization \eqref{galR} can now be expressed in terms of
$\bar u_{ij}^s$ and
\beq\label{fij}
f_{ij}^s=2d_{ij}\left[\frac{u_i-u_j}2-
\alpha_{ij}P_{ij}
  \right]+a_{ij}^R(u_i-u_j)
\eeq
as follows:
\beq\label{galerkin-wb}
a_i^Ru_i-\sum_{j\in\Nis}[2d_{ij}(\bar u_{ij}^s-u_i)+f_{ij}^s -a_{ij}^D(u_j-u_i)]=0.
\eeq
Note that we have distributed the source term $b_i$ among the bar states $\bar u_{ij}^s$
and stabilized these intermediate
states using the limited balancing fluxes $\alpha_{ij}P_{ij}$. A similar algebraic splitting
was used in \cite{hajduk2022diss,hajduk2022} to construct a well-balanced MCL scheme
for the SWE system. Our definition of 
$f_{ij}^s$ ensures that $f_{ij}^s=0$ if the coefficients of problem \eqref{bvp}
are given by \eqref{constcoeff}, $u_h=\hat u$, and $\alpha_{ij}=1$. This
enables us to preserve strong consistency at the corresponding steady state
(see Remark \ref{remark:WB} below).

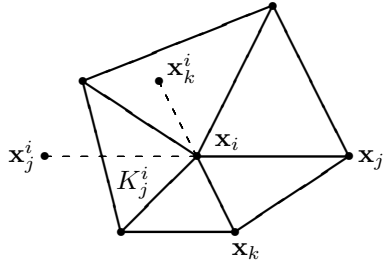
\begin{figure}[b]
\begin{center}
\setlength{\unitlength}{1cm}
\begin{picture}(4.7,3)(-0.35,1)
\thinlines

\dashline{0.1}[0.001](2,2)(1.5,3)
\dashline{0.1}[0.001](2,2)(0,2)

\thicklines

\put(2,2){\circle*{0.1}}
\put(4,2){\circle*{0.1}}
\put(3,4){\circle*{0.1}}
\put(0.5,3){\circle*{0.1}}
\put(1,1){\circle*{0.1}}
\put(2.5,1){\circle*{0.1}}
\put(1.5,3){\circle*{0.1}}
\put(0,2){\circle*{0.1}}
\drawline(4,2)(3,4)(0.5,3)(1,1)(2.5,1)(4,2)
\drawline(4,2)(2,2)(3,4)
\drawline(0.5,3)(2,2)(1,1)
\drawline(2,2)(2.5,1)

\put(2.4,2.19){\makebox(0,0){$\bx_i$}}
\put(2.65,0.75){\makebox(0,0){$\bx_k$}}
\put(4.3,1.95){\makebox(0,0){$\bx_j$}}
\put(-0.3,2){\makebox(0,0){$\fictxj$}}
\put(1.8,3.17){\makebox(0,0){$\fictxk$}}
\put(1.16,1.65){\makebox(0,0){$\cellij$}}
\end{picture}
\end{center}
\caption{Fictitious nodes.}
\label{fig:fict_node}
\end{figure}

To design an algorithm that produces $\alpha_{ij}=1$ for
$u_h=\hat u$ given by formula \eqref{uhatdef}, we introduce
fictitious nodes $\fictxj$ which are placed symmetrically to the nodes
$\bx_j$, $j\in\Nis$ with respect to $\bx_i$, i.e., $(\fictxj+\bx_j)/2=\bx_i$;
see Fig.~\ref{fig:fict_node}. We denote by $\fictuj$ the (fictitious) value of
$u_h$ at $\fictxj$. If the fictitious node is contained in one of the mesh
cells containing $\bx_i$ (like the node $\fictxk$ in Fig.~\ref{fig:fict_node}),
then we simply evaluate $u_h$ at this point. If this is not the case, 
then following \cite{Kno23}
we denote by $\cellij$ a mesh cell containing $\bx_i$ that is intersected by
the half line $\{\bx_i+\theta\,(\bx_i-\bx_j)\,:\,\,\theta>0\}$
(cf.~Fig.~\ref{fig:fict_node}), extend $u_h|_{\cellij}^{}$ to a first degree 
polynomial on $\mathbb{R}^d$ and evaluate this extension at
$\fictxj$. Thus, in both cases, we obtain
\begin{equation}\label{eq:fictuj}
   \fictuj=u_i+\nabla u_h|_{\cellij}^{}\cdot(\fictxj-\bx_i)
          =u_i+\nabla u_h|_{\cellij}^{}\cdot(\bx_i-\bx_j).
\end{equation}
Other definitions of values at fictitious nodes can be found e.g.~in
\cite{AD93,LMPP94,BB17}.

\begin{remark}
If $|\mathbf{v}(\mathbf{x}_i)|+|\mathbf{v}(\mathbf{x}_j)|$ is small,
then the magnitude of $P_{ij}$ may become large. In \eqref{fij} and
\eqref{galerkin-wb}, this is compensated by the multiplication by $d_{ij}$ 
that depends on $\mathbf{v}$.
\end{remark}

Let us now proceed to formulating appropriate
inequality constraints for well-balanced flux limiting.
The multiplication by the correction factor
$\alpha_{ij}=\alpha_{ji}$ in the formula for the flux $f_{ij}^s=-f_{ji}^s$
makes it possible to enforce the local discrete maximum principles
\begin{subequations}\label{dmpbar}
\begin{align}\label{eq:dmpbar1}
 u_i=u_i^{\max}\ \,\wedge\ \, & b_i\le 0\quad\Rightarrow\quad \bar u_{ij}^s\le \max\{u_i,u_j\}
 {\quad\forall\,\,j\in\Nis,}\\
 u_i=u_i^{\min}\ \,\wedge\ \, & b_i\ge 0\quad\Rightarrow\quad \bar u_{ij}^s\ge \min\{u_i,u_j\}
 {\quad\forall\,\,j\in\Nis,}\label{eq:dmpbar2}
\end{align}
\end{subequations}
{for $i=1,\dots,M_h$.} Adopting this design criterion, we use the auxiliary 
quantities
\begin{equation*}
  \begin{array}{l}
   \displaystyle
   Q_{ij}^+=\max\left\{\frac{\fictuj-u_i}2,\max\{u_i,u_j\}-\bar u_{ij}-\frac{b_i}{a_i^C}\right\},\\[5mm]
   \displaystyle
   Q_{ij}^-=\min\left\{\frac{\fictuj-u_i}2,\min\{u_i,u_j\}-\bar u_{ij}-\frac{b_i}{a_i^C}\right\},
   \end{array}\qquad i=1,\dots,M_h,\quad j\in\Nis
\end{equation*}
 to define
\begin{equation*}
  R_{ij}=\begin{cases}
  \displaystyle
  \frac{Q_{ij}^+}{P_{ij}} & \mbox{if}\ b_i\le 0\ \mbox{and}\ P_{ij}>Q_{ij}^+,\\[3mm]
  \displaystyle
  \frac{Q_{ij}^-}{P_{ij}} & \mbox{if}\ b_i\ge 0\ \mbox{and}\ P_{ij}<Q_{ij}^-,\\[3mm]
  1 & \mbox{otherwise},
  \end{cases}
  \qquad{i=1,\dots,M_h,\quad j\in\Nis.}
\end{equation*}
{Furthermore, we set
\begin{equation*}
  R_{ij}=1,\qquad i=M_h+1,\dots,N_h,\quad j\in\Nis.
\end{equation*}
Then we define}
\begin{equation*}
   \alpha_{ij}=\min\{R_{ij},R_{ji}\},\qquad
   {i=1,\dots,N_h,\quad j\in\Nis.}
\end{equation*}
This limiting strategy yields $\alpha_{ij}\in[0,1]$ such that
$\alpha_{ij}=\alpha_{ji}$ and conditions \eqref{dmpbar} are satisfied.

\begin{remark}
In \eqref{dmpbar}, both $\max\{u_i,u_j\}$ and $\min\{u_i,u_j\}$ are equal to $u_i$
but we use the present formulation to establish a correspondence to the
definition of $Q_{ij}^\pm$. Note that, under the sign conditions on $b_i$, the
left-hand side inequalities of \eqref{dmpbar} hold not only under the
assumption that $u_i$ is a local maximum or minimum, but also when
$(u^i_j-u_i)/2$ is dominated by the second term in the definition of $Q_{ij}^+$
or $Q_{ij}^-$. In particular, this is the case if $u^i_j-u_i$ has the same sign
as $b_i$. Replacing $\max\{u_i,u_j\}$ and $\min\{u_i,u_j\}$ by $u_i$ in the
definitions of $Q_{ij}^\pm$ and in \eqref{dmpbar} would be too restrictive
since then the left-hand side inequalities of \eqref{dmpbar} would hold for a much smaller
class of functions.
\end{remark}

In a practical implementation, we calculate the limited balancing fluxes
\begin{equation}\label{eq:alpha_rel1}
   \alpha_{ij}P_{ij}
   =\mathrm{sgn}(P_{ij})\min\{R_{ij}|P_{ij}|,R_{ji}|P_{ji}|\},\qquad
   {i=1,\dots,N_h,\quad j\in\Nis}
\end{equation}
directly to avoid possible division by zero in finite precision arithmetic.
Note that
$Q_{ij}^+\ge 0$ if $b_i\le 0$ and 
$Q_{ij}^-\le 0$ if $b_i\ge 0$. It follows that 
\begin{equation}\label{eq:alpha_rel2}
  R_{ij}|P_{ij}|=\begin{cases}
  \mathrm{sgn}(P_{ij})\min\{P_{ij},Q_{ij}^+\}& \mbox{if}\ b_i<0\ 
  \mbox{or}\ (b_i=0\ \mbox{and}\ P_{ij}\ge0),\\
  \mathrm{sgn}(P_{ij})\max\{P_{ij},Q_{ij}^-\}& \mbox{if}\ b_i>0\
  \mbox{or}\ (b_i=0\ \mbox{and}\ P_{ij}\le0)
  \end{cases}
\end{equation}
for $i=1,\dots,M_h$ and $j\in\Nis$.

\smallskip

In the context of scalar CDR problems, we define the limited approximation
\begin{equation}\label{eq:limited_flux}
f_{ij}^{s,*} =
	\begin{cases} 
	\min\left\lbrace f_{ij}^s, \min \left\lbrace 2d_{ij}
\left(\bar u_i^{\max}-\bar{u}_{ij}^s\right), 2d_{ij}
\left(\bar{u}_{ji}^s-\bar u_j^{\min}\right) \right\rbrace \right\rbrace &\text{if }
f_{ij}^s>0, \\
	0 &\text{if } f_{ij}^s=0, \\
	\max\left\lbrace f_{ij}^s, \max \left\lbrace 2d_{ij}
\left(\bar u_i^{\min}-\bar{u}_{ij}^s\right), 2d_{ij}
\left(\bar{u}_{ji}^s-\bar u_j^{\max}\right)\right\rbrace \right\rbrace &\text{if }
f_{ij}^s<0
	\end{cases}
\end{equation}
to $f_{ij}^s$ using the low-order bar states $\bar u_{ij}^s$ to construct 
the local bounds
$$
\bar u_i^{\min}:=\min_{j\in\Nis}\bar u_{ij}^s,\qquad
\bar u_i^{\max}:=\max_{j\in\Nis}\bar u_{ij}^s.
$$
This limiting strategy ensures that
the flux-corrected intermediate states 
$$\bar u_{ij}^{s,*}=\bar u_{ij}^s+\frac{f_{ij}^{s,*}}{2d_{ij}}$$
satisfy the inequality constraints
\begin{equation}\label{eq:constraints}
   \bar u_i^{\min} \le \bar u_{ij}^{s,*}\le \bar u_i^{\max}.
\end{equation}
{Note that \eqref{eq:limited_flux} is well defined only if both
indices $i$ and $j$ refer to interior nodes (and $j\in\Nis$ as usual). If 
$i\in\{1,\dots,M_h\}$ and $j\in\Ni\cap\{M_h+1,\dots,N_h\}$, we set}
\begin{equation}\label{eq:limited_flux2}
f_{ij}^{s,*} =
	\begin{cases} 
	\min\left\lbrace f_{ij}^s, 
        2d_{ij} \left(\bar u_i^{\max}-\bar{u}_{ij}^s\right) \right\rbrace 
        &\text{if } f_{ij}^s>0, \\
	0 &\text{if } f_{ij}^s=0, \\
	\max\left\lbrace f_{ij}^s, 
        2d_{ij} \left(\bar u_i^{\min}-\bar{u}_{ij}^s\right) \right\rbrace 
        &\text{if } f_{ij}^s<0.
	\end{cases}
\end{equation}
{This again guarantees that the constraints
\eqref{eq:constraints} hold.}
\smallskip

Our well-balanced MCL scheme for problem \eqref{bvp} can be written in the `homogeneous' form
\begin{subequations}\label{wmc}
  \begin{alignat}{3}\label{mcl:bar}
a_i^Ru_i-\sum_{j\in\Nis}[2d_{ij}(\bar u_{ij}^{s,*}-u_i)-a_{ij}^D(u_j-u_i)]&=0,
&& {i=1,\dots,M_h,}\\
  {u_i}&{=u_D(\mathbf{x}_i),}\qquad&& {i=M_h+1,\dots,N_h},\label{eq:wmc_bc}
\end{alignat}
\end{subequations}
in which the source terms are incorporated 
into the bar states $\bar u_{ij}^{s,*}$ (similarly to
 \cite{hajduk2022diss,hajduk2022}).

\begin{remark}
In practice, $u_i$ can be calculated using the
bound-preserving fixed-point iteration
$$
u_i^{n+1}=\frac{1}{a_i}
\sum_{j\in\Nis}[2d_{ij}\bar u_{ij}^{s,*,n}-a_{ij}^Du_j^n],
$$
where 
$a_i=a_i^R+a_i^C-\sum_{j\in\Nis}a_{ij}^D>0$. Each update produces a linear combination
of the states that appear on the right-hand side.
In view of our assumption that
$a_{ij}^D\le 0$ for $j\ne i$, all weights are nonnegative. Moreover, they
add up to unity if $a_i^R=0$.
\end{remark}

\begin{remark}\label{remark:WB}
If the coefficients of problem \eqref{bvp} are given by \eqref{constcoeff},
then $s_i=s_j=\hat f$. Moreover, $c\equiv 0$ implies $a_{ij}^R=0$. Substituting
the nodal values $u_i=\hat u(\mathbf{x}_i)$ of the exact steady-state solution
\eqref{uhatdef} into the definition \eqref{balflux} of the balancing flux, we
find that $P_{ij}=(u_i-u_j)/2$. In addition, since $u_h$ is a first degree
polynomial in $\Omega$, one has $\fictuj-u_i=u_i-u_j$ due to \eqref{eq:fictuj}.
Hence, by definition of $Q_{ij}^\pm$, it follows that $\alpha_{ij}=1$ and
$f_{ij}^{s}=0=f_{ij}^{s,*}$. Therefore, the flux-corrected scheme
\eqref{mcl:bar} coincides with the equilibrium-preserving Galerkin
discretization \eqref{galerkin-wb}. This proves that \eqref{mcl:bar} is well
balanced.
\end{remark}

\begin{remark}\label{remark:simpler}
  To avoid the evaluation of $u_h$ at fictitious nodes, a simplified
  version of the above algorithm calculates the correction factors $\alpha_{ij}$ using
\begin{equation*}
   Q_{ij}^+=\max\{u_i,u_j\}-\bar u_{ij}-\frac{b_i}{a_i^C},\qquad
   Q_{ij}^-=\min\{u_i,u_j\}-\bar u_{ij}-\frac{b_i}{a_i^C},
\end{equation*}
for $i=1,\dots,M_h$, $j\in\Nis$. Then, instead of \eqref{dmpbar}, one has the stronger properties
\begin{align*}
 & b_i\le 0\quad\Rightarrow\quad \bar u_{ij}^s\le \max\{u_i,u_j\}
 \quad\forall\,\,j\in\Nis,\\
 & b_i\ge 0\quad\Rightarrow\quad \bar u_{ij}^s\ge \min\{u_i,u_j\}
 \quad\forall\,\,j\in\Nis,
\end{align*}
for $i=1,\dots,M_h$. The theoretical results that we prove in the
next section remain valid. However, the assertion of Remark~\ref{remark:WB} is
not true in general any more for this version of MCL. Nevertheless,
one can still prove that the scheme \eqref{mcl:bar} is well balanced on some
types of uniform meshes. On general meshes, the scheme may be not
well balanced, as numerical experiments show.
\end{remark}

\section{Solvability and discrete maximum principle}\label{sec:analysis}

In this section, we investigate
the solvability of the nonlinear problem \eqref{wmc} and the
validity of local and global discrete maximum principles. All results will be proven under the assumption that
$\varepsilon>0$. To prove the solvability, additional assumptions on the data
will be made as well.

First we cast \eqref{mcl:bar} into a form that will be convenient for our
analysis. It follows from \eqref{eq:ubar} that
\begin{equation}\label{eq:ub2}
   d_{ij}(u_i-u_j)=2d_{ij}(u_i-\bar u_{ij})+a_{ij}^C(u_i-u_j)\,.
\end{equation}
Substituting \eqref{eq:ub2} into \eqref{fij}, one obtains
\begin{equation}\label{eq:fijq2}
   f_{ij}^s=2d_{ij}(u_i-\bar u_{ij}^s)+(a_{ij}^C+a_{ij}^R)(u_i-u_j)
            +2d_{ij}\frac{b_i}{a_i^C}.
\end{equation}
Using this expression, it is easy to verify that \eqref{mcl:bar} can be
equivalently written in the form
\begin{equation}\label{wmc2}
   a_i^Ru_i+\sum_{j\in\Nis}(a_{ij}^D+a_{ij}^C+a_{ij}^R)(u_j-u_i)
   +\sum_{j\in\Nis}(f_{ij}^s-f_{ij}^{s,*})=b_i,\quad i=1,\dots,M_h.
\end{equation}

The solvability proof will be based on the following consequence of the
Brouwer fixed-point theorem.

\begin{lemma}\label{lem:Temam}
Let $X$ be a finite-dimensional Hilbert space with inner product
$(\cdot,\cdot)_X$ and norm $\|\cdot\|_X^{}$. Let $T:X\to X$ be a continuous
mapping and $K>0$ a real number such that $(Tx,x)_X>0$ for any $x\in X$ with
$\|x\|_X^{}=K$.  Then there exists $x\in X$ such that $\|x\|_X^{}\le K$ and 
$Tx=0$.
\end{lemma}

\begin{proof}See \cite[p.~164, Lemma 1.4]{Te77}.
\end{proof}

\begin{theorem}
Let the data of \eqref{bvp} satisfy $\varepsilon>0$, $\nabla\cdot\mathbf{v}=0$,
and $c\equiv0$. Then the nonlinear problem \eqref{wmc} has a solution.
\end{theorem}

\begin{proof}
The fluxes $f_{ij}^s$ and $f_{ij}^{s,*}$ are functions of the coefficient
vectors $u:=(u_1,\dots,u_{N_h})^\top\in{\mathbb R}^{N_h}$. Let us first
investigate whether they depend on $u$ in a continuous way. We will proceed
step by step. To show that the bar states $\bar u_{ij}^s$ are continuous
functions of $u$, it suffices to investigate the continuity of
\eqref{eq:alpha_rel1}. Since minimum and maximum are continuous functions, the
functions $Q_{ij}^+$ and $Q_{ij}^-$ are continuous. If 
$\bar u\in{\mathbb R}^{N_h}$ is such that $P_{ij}(\bar u)\neq0$, then
$P_{ij}\neq0$ in a neighborhood of $\bar u$ and hence $R_{ij}|P_{ij}|$ and
$R_{ji}|P_{ji}|$ are continuous in this neighborhood in view of
\eqref{eq:alpha_rel2}. Thus, $\alpha_{ij}P_{ij}$ is continuous at $\bar u$ due
to \eqref{eq:alpha_rel1}. Moreover, if $P_{ij}(\bar u)=0$, one obtains
\begin{equation}\label{eq:cont_est}
   \left|(\alpha_{ij}P_{ij})(u)-(\alpha_{ij}P_{ij})(\bar u)\right|
   =|\alpha_{ij}P_{ij}|(u)\le|P_{ij}(u)|=|P_{ij}(u)-P_{ij}(\bar u)|
\end{equation}
so that $\alpha_{ij}P_{ij}$ is continuous at $\bar u$ also in this case.
Therefore, the function in \eqref{eq:alpha_rel1} is continuous on ${\mathbb
R}^{N_h}$. Consequently, $\bar u_{ij}^s$, $f_{ij}^s$, $\bar u_i^{\min}$, and
$\bar u_i^{\max}$ are continuous on ${\mathbb R}^{N_h}$. Then, if
$f_{ij}^s(\bar u)\neq0$ for some $\bar u\in{\mathbb R}^{N_h}$, it follows from
\eqref{eq:limited_flux} and \eqref{eq:limited_flux2} that $f_{ij}^{s,*}$ is
continuous in a neighborhood of $\bar u$. If $f_{ij}^s(\bar u)=0$, then the
continuity of $f_{ij}^{s,*}$ at $\bar u$ follows as in \eqref{eq:cont_est}
since $|f_{ij}^{s,*}|\le|f_{ij}^s|$ due to \eqref{eq:limited_flux} and
\eqref{eq:limited_flux2}.

A coefficient vector $u=(u_1,\dots,u_{N_h})^\top$ solving \eqref{wmc} can be split
into the vectors $u_{\rm I}:=(u_1,\dots,u_{M_h})^\top$ and 
$u_{\rm B}:=(u_{M_h+1},\dots,u_{N_h})^\top
=(u_D(\mathbf{x}_{M_h+1}),\dots,u_D(\mathbf{x}_{N_h}))^\top$. Let us define a
mapping $T:{\mathbb R}^{M_h}\to{\mathbb R}^{M_h}$ by
\begin{equation*}
   (Tv)_i=\sum_{j=1}^{M_h}(a_{ij}^D+a_{ij}^C+a_{ij}^R)v_j
   +\sum_{j\in\Nis}(f_{ij}^s-f_{ij}^{s,*})(v,u_{\rm B})-g_i,\quad 
   i=1,\dots,M_h,\quad v\in{\mathbb R}^{M_h},
\end{equation*}
where
\begin{equation*}
   g_i=b_i-\sum_{j=M_h+1}^{N_h}(a_{ij}^D+a_{ij}^C+a_{ij}^R)u_D(\mathbf{x}_j),
   \quad i=1,\dots,M_h.
\end{equation*}
Then $T$ is continuous and, since \eqref{mcl:bar} and \eqref{wmc2} are 
equivalent, a vector $u\in{\mathbb R}^{N_h}$ satisfying \eqref{eq:wmc_bc} 
solves \eqref{mcl:bar} if and only if $Tu_{\rm I}=0$. Thus, in view of
Lemma~\ref{lem:Temam}, to prove the solvability of \eqref{wmc}, it suffices to
analyze the product $(Tv,v)$, where $(\cdot,\cdot)$ denotes the Euclidean
inner product on ${\mathbb R}^{M_h}$.

Introducing
\begin{equation*}
   v_h=\sum_{j=1}^{M_h}v_j\varphi_j
\end{equation*}
and using the assumptions of the theorem, we deduce that
\begin{align*}
   \sum_{i,j=1}^{M_h}v_i(a_{ij}^D+a_{ij}^C+a_{ij}^R)v_j
   &=\int_\Omega\left(\varepsilon|\nabla v_h|^2
   + v_h[  \mathbf{v}\cdot \nabla v_h+  c\,v_h]\right)\dx\\
   &=\varepsilon\|v_h\|_{H^1_0(\Omega)}^2
   +\frac12\int_\Omega\nabla\cdot(\mathbf{v}v_h^2)\dx
   =\varepsilon\|v_h\|_{H^1_0(\Omega)}^2.
\end{align*}
Thus, it follows from the equivalence of norms on finite-dimensional spaces
that there is a positive constant $C_1$ independent of $v$ such that
\begin{equation}\label{eq:est1}
   \sum_{i,j=1}^{M_h}v_i(a_{ij}^D+a_{ij}^C+a_{ij}^R)v_j\ge C_1\|v\|^2,
\end{equation}
where $\|\cdot\|$ is the Euclidean norm on ${\mathbb R}^{M_h}$. To estimate the
term with the fluxes, let us first introduce a matrix $(b_{ij})_{i,j=1}^{N_h}$
of correction factors 
\begin{alignat*}{2}
  &b_{ij}=\begin{cases}
  \frac{\displaystyle f_{ij}^s-f_{ij}^{s,*}}{\displaystyle f_{ij}^s}& 
  \mbox{if}\ j\in\Nis \ \mbox{and} \ f_{ij}^s\neq0,\\
  0\hspace*{15mm} & \mbox{otherwise},
  \end{cases}\qquad&&i=1,\dots,M_h,\\[1mm]
  &b_{ij}=\begin{cases}
  b_{ji} & \mbox{for} \ j=1,\dots,M_h,\\
  0\hspace*{15mm} & \mbox{otherwise},
  \end{cases}\qquad&&i=M_h+1,\dots,N_h.
\end{alignat*}
Then
\begin{equation*}
   b_{ij}=b_{ji}\quad\mbox{and}\quad b_{ij}\in[0,1]\qquad
   \forall\,\,i,j=1,\dots,N_h.
\end{equation*}
Moreover,
\begin{equation*}
   \sum_{j\in\Nis}(f_{ij}^s-f_{ij}^{s,*})(v,u_{\rm B})
   =\sum_{j\in\Nis}(b_{ij}f_{ij}^s)(v,u_{\rm B}),
   \quad i=1,\dots,M_h.
\end{equation*}
Denoting $z:=(v,u_{\rm B})$, one has
$$
   f_{ij}^s(v,u_{\rm B})=d_{ij}(z_i-z_j)-2d_{ij}\alpha_{ij}(z)P_{ij},
$$
where $P_{ij}$ is independent of $z$ since $c\equiv0$. Thus,
\begin{align*}
   &\sum_{i=1}^{M_h}\sum_{j\in\Nis}v_i(f_{ij}^s-f_{ij}^{s,*})(v,u_{\rm B})\\
   &\hspace*{20mm}
   =\sum_{i=1}^{M_h}\sum_{j=1}^{N_h}b_{ij}(z)d_{ij}(z_i-z_j)z_i
   -2\sum_{i=1}^{M_h}\sum_{j\in\Nis}d_{ij}(b_{ij}\alpha_{ij})(z)P_{ij}v_i=I-II
\end{align*}
with
\begin{align*}
   I&=\sum_{i,j=1}^{N_h}b_{ij}(z)d_{ij}(z_i-z_j)z_i,\\
   II&=\sum_{i=M_h+1}^{N_h}\sum_{j=1}^{N_h}b_{ij}(z)d_{ij}(z_i-z_j)z_i
   +2\sum_{i=1}^{M_h}\sum_{j\in\Nis}d_{ij}(b_{ij}\alpha_{ij})(z)P_{ij}v_i.
\end{align*}
Interchanging $i$ and $j$ in the formula defining $I$ and using the fact that
$b_{ji}(z)d_{ji}=b_{ij}(z)d_{ij}$, one obtains
\begin{equation*}
   I=\sum_{i,j=1}^{N_h}b_{ij}(z)d_{ij}(z_j-z_i)z_j,
\end{equation*}
which implies that
\begin{equation}\label{eq:est2}
   I=\frac12\sum_{i,j=1}^{N_h}b_{ij}(z)d_{ij}(z_i-z_j)^2\ge0.
\end{equation}
Furthermore,
\begin{align*}
   II&=\sum_{i,j=M_h+1}^{N_h}b_{ij}(z)d_{ij}
             (u_D(\mathbf{x}_i)-u_D(\mathbf{x}_j))u_D(\mathbf{x}_i)
   +\sum_{i=M_h+1}^{N_h}\sum_{j=1}^{M_h}b_{ij}(z)d_{ij}u_D(\mathbf{x}_i)^2\\
   &\hspace*{17mm}
   -\sum_{i=M_h+1}^{N_h}\sum_{j=1}^{M_h}b_{ij}(z)d_{ij}u_D(\mathbf{x}_i)v_j
   +2\sum_{i=1}^{M_h}\sum_{j\in\Nis}d_{ij}(b_{ij}\alpha_{ij})(z)P_{ij}v_i
\end{align*}
and hence there are positive constants $C_2$ and $C_3$ independent of $v$ such 
that
\begin{equation}\label{eq:est3}
   |II|\le C_2\|v\|+C_3.
\end{equation}
Combining \eqref{eq:est1}--\eqref{eq:est3} and applying the Cauchy--Schwarz
inequality, one obtains
\begin{equation*}
   (Tv,v)\ge C_1\|v\|^2-C_2\|v\|-C_3-\|g\|\|v\|\qquad
   \forall\,\,v\in{\mathbb R}^{M_h}
\end{equation*}
with $g:=(g_1,\dots,g_{M_h})^\top$. Thus, for any
$K>\max\{1,(C_2+C_3+\|g\|)/C_1\}$, one has $(Tv,v)>0$ for all $v\in{\mathbb
R}^{M_h}$ with $\|v\|=K$, and the assertion of the theorem is true by Lemma~\ref{lem:Temam}.
\end{proof}

The following result will be useful for proving local and global DMPs.

\begin{lemma}\label{lemmma1}
For any vector $(u_1,\dots,u_{N_h})^\top\in{\mathbb R}^{N_h}$ and any pair
of indices $i\in\{1,\dots,M_h\}$, $j\in\Nis$, the following estimates hold:
\begin{equation}\label{eq:bounds}
   2d_{ij}(u_i-\bar u_i^{\max})
   \le(a_{ij}^C+a_{ij}^R)(u_j-u_i)-2d_{ij}\frac{b_i}{a_i^C}
   +f_{ij}^s-f_{ij}^{s,*}\le2d_{ij}(u_i-\bar u_i^{\min}).
\end{equation}
\end{lemma}

\begin{proof} Denote
\begin{equation*}
   q_{ij}:=(a_{ij}^C+a_{ij}^R)(u_j-u_i)-2d_{ij}\frac{b_i}{a_i^C}
   +f_{ij}^s-f_{ij}^{s,*}.
\end{equation*}
Using \eqref{eq:fijq2}, one obtains
\begin{equation*}
   q_{ij}=2d_{ij}(u_i-\bar u_{ij}^s)-f_{ij}^{s,*}.
\end{equation*}
If $f_{ij}^s>0$, then, according to \eqref{eq:limited_flux} and
\eqref{eq:limited_flux2}, one has
\begin{equation*}
   f_{ij}^{s,*}\le 2d_{ij}(\bar u_i^{\max}-\bar{u}_{ij}^s)
\end{equation*}
and hence
\begin{equation*}
   q_{ij}\ge2d_{ij}(u_i-\bar u_i^{\max}).
\end{equation*}
If $f_{ij}^s\le0$, then $f_{ij}^{s,*}\le0$ and hence
\begin{equation*}
   q_{ij}\ge2d_{ij}(u_i-\bar u_{ij}^s)\ge2d_{ij}(u_i-\bar u_i^{\max}).
\end{equation*}
This proves the first inequality in \eqref{eq:bounds}. The second one follows analogously.
\end{proof}

\begin{theorem}
Let $\varepsilon>0$. Then the solution of \eqref{wmc} satisfies the following
local DMPs for any $i\in\{1,\dots,M_h\}$:
\begin{subequations}\label{ldmp1}
\begin{align}\label{ldmp_max1}
 & b_i\le 0\quad\Rightarrow\quad u_i\le \max_{j\in\Nis}u_j^+,\\
 & b_i\ge 0\quad\Rightarrow\quad u_i\ge \min_{j\in\Nis}u_j^-,
 \label{ldmp_min1}
\end{align}
\end{subequations}
where $u_j^+=\max\{u_j,0\}$ and $u_j^-=\min\{u_j,0\}$. If $a_i^R=0$, then the
following stronger local DMPs hold:
\begin{subequations}\label{ldmp2}
\begin{align}\label{ldmp_max2}
 & b_i\le 0\quad\Rightarrow\quad u_i\le \max_{j\in\Nis}u_j,\\
 & b_i\ge 0\quad\Rightarrow\quad u_i\ge \min_{j\in\Nis}u_j.
 \label{ldmp_min2}
\end{align}
\end{subequations}
\end{theorem}

\begin{proof}
Consider any $i\in\{1,\dots,M_h\}$ such that $b_i\le0$. If $a_i^R\neq0$, it
suffices to assume that $u_i>0$ since otherwise \eqref{ldmp_max1} holds trivially. Hence $a_i^Ru_i\ge0$ since $a_i^R\ge0$. Since the solution of
\eqref{wmc} satisfies \eqref{wmc2}, it follows from \eqref{eq:bounds} that
\begin{equation}\label{eq1}
   b_i\ge\sum_{j\in\Nis}\left\{a_{ij}^D(u_j-u_i)+
   2d_{ij}(u_i-\bar u_i^{\max})+2d_{ij}\frac{b_i}{a_i^C}\right\}.
\end{equation}
Let us assume that $u_i>u_j$ for all $j\in\Nis$. Then 
$u_i\ge\bar u_i^{\max}$ due to \eqref{eq:dmpbar1} and it follows from
\eqref{eq1} that
\begin{equation}\label{eq2}
   b_i\ge b_i+\sum_{j\in\Nis}a_{ij}^D(u_j-u_i).
\end{equation}
Since $a_{ij}^D\le0$ for any $j\in\Nis$, $a_{ii}^D>0$, and
$\sum_{j\in\Ni}a_{ij}^D=0$, there exists $j\in\Nis$ such that $a_{ij}^D<0$.
Therefore,
\begin{equation*}
   \sum_{j\in\Nis}a_{ij}^D(u_j-u_i)>0,
\end{equation*}
which is in contradiction to \eqref{eq2}. Consequently, there exists $j\in\Nis$
such that $u_i\le u_j$, thus proving \eqref{ldmp_max2} and hence also
\eqref{ldmp_max1}.

The implications \eqref{ldmp_min1} and \eqref{ldmp_min2} follow analogously.
\end{proof}

To prove global DMPs, we assume that the mesh is such that, for any 
$i\in\{1,\dots,M_h\}$, there exist $k\in\{M_h+1,\dots,N_h\}$ and
$i_1,i_2,\dots,i_l\in\{1,\dots,M_h\}$ such that all these indices are mutually
different and 
\begin{equation}\label{eq:connect}
   a_{ii_1}^D\ne0,\quad a_{i_1i_2}^D\ne0,\quad\dots\quad 
   a_{i_{l-1}i_l}^D\ne0,\quad a_{i_lk}^D\ne0.
\end{equation}
This assumption is typically satisfied.

\begin{theorem}
Let $\varepsilon>0$. Then the solution of \eqref{wmc} satisfies the following
global DMPs:
\begin{subequations}\label{gdmp1}
\begin{align}\label{gdmp_max1}
 &b_i\le0,\,\,i=1,\dots,M_h\,\,\,\Rightarrow\,\,\,
   \max_{i=1,\dots,N_h}\,u_i\le\max_{i=M_h+1,\dots,N_h}\,u_i^+\,,\\
 &b_i\ge0,\,\,i=1,\dots,M_h\,\,\,\Rightarrow\,\,\,
   \min_{i=1,\dots,N_h}\,u_i\ge\min_{i=M_h+1,\dots,N_h}\,u_i^-\,.
 \label{gdmp_min1}
\end{align}
\end{subequations}
If $c=0$ in $\Omega$, then the following stronger global DMPs hold:
\begin{subequations}\label{gdmp2}
\begin{align}\label{gdmp_max2}
 &b_i\le0,\,\,i=1,\dots,M_h\,\,\,\Rightarrow\,\,\, 
   \max_{i=1,\dots,N_h}\,u_i=\max_{i=M_h+1,\dots,N_h}\,u_i\,,\\
 &b_i\ge0,\,\,i=1,\dots,M_h\,\,\,\Rightarrow\,\,\,
   \min_{i=1,\dots,N_h}\,u_i=\min_{i=M_h+1,\dots,N_h}\,u_i\,.
 \label{gdmp_min2}
\end{align}
\end{subequations}
\end{theorem}

\begin{proof}
Let us assume that $b_i\le0$ for all $i=1,\dots,M_h$. If $c$ does not vanish in
$\Omega$, it suffices to assume that $\max_{i=1,\dots,N_h}\,u_i>0$ since
otherwise \eqref{gdmp_max1} holds trivially. In this case, the right-hand side
of the implication \eqref{gdmp_max1} reduces to the right-hand side of the
implication \eqref{gdmp_max2}.

Let $i\in\{1,\dots,N_h\}$ be an arbitrary index such that
\begin{equation}\label{eq:ui_max}
   u_i=\max_{j=1,\dots,N_h}\,u_j\,.
\end{equation}
If $i\in\{M_h+1,\dots,N_h\}$, then the right-hand side equality of the implication
\eqref{gdmp_max2} holds. Thus, let us assume that $i\in\{1,\dots,M_h\}$. Since
$a_i^Ru_i\ge0$, the inequality \eqref{eq1} holds again. Using
\eqref{eq:dmpbar1}, one obtains $u_i\ge\bar u_i^{\max}$ and hence it follows
from \eqref{eq1} that
\begin{equation*}
   0\ge\sum_{j\in\Nis}a_{ij}^D(u_j-u_i).
\end{equation*}
Since $a_{ij}^D\le0$ for any $j\in\Nis$ and $u_i$ is a global maximum, all
terms in the sum are nonnegative, which implies that
\begin{equation*}
   a_{ij}^D(u_j-u_i)=0\quad\forall\,\,j\in\Nis.
\end{equation*}
Let $k\in\{M_h+1,\dots,N_h\}$ and $i_1,i_2,\dots,i_l\in\{1,\dots,M_h\}$ be such
that \eqref{eq:connect} holds. Then $u_{i_1}=u_i$ and hence \eqref{eq:ui_max}
holds with $i=i_1$.  Repeating the above arguments, one finally concludes that
\eqref{eq:ui_max} holds with $i=k\in\{M_h+1,\dots,N_h\}$, which proves that the
right-hand side equality of the implication \eqref{gdmp_max2} holds.

The proof of \eqref{gdmp_min1} and \eqref{gdmp_min2} is analogous.
\end{proof}

The global DMPs imply that our well-balanced MCL scheme is positivity
preserving.

\begin{corollary}
  Let $\varepsilon>0$. Consider a finite element approximation $u_h$
 of the form \eqref{uhdef}. Suppose that its coefficients satisfy \eqref{wmc}. Then
\begin{equation*}
   f\ge0\ \mbox{in}\ \Omega\quad\mbox{and}\quad u_D\ge0\ \mbox{on}\ \Gamma_D
   \quad\Rightarrow\quad u_h\ge0\ \mbox{in}\ \Omega.
\end{equation*}
\end{corollary}

\begin{proof}
If $f\ge0$ in $\Omega$, then $b_i\ge0$ for $i=1,\dots,M_h$ and hence it follows
from \eqref{gdmp_min1} that the solution of \eqref{wmc} satisfies
\begin{equation*}
   \min_{i=1,\dots,N_h}\,u_i\ge\min_{i=M_h+1,\dots,N_h}\,u_i^-
   =\min_{i=M_h+1,\dots,N_h}\,\min\{u_D(\mathbf{x}_i),0\}.
\end{equation*}
Thus, if $u_D\ge0$ on $\Gamma_D$, one has $u_i\ge0$ for $i=1,\dots,N_h$. Since
the minimum of a continuous  function that is piecewise linear
on $\mathcal{T}_h$ is
attained at a vertex of $\mathcal{T}_h$, it follows that $u_h\ge0$ in $\Omega$.
\end{proof}

\section{Numerical examples}\label{sec:numres}

\input{numres.tex}

\section{Conclusions}\label{sec:summary}

This paper demonstrates that flux correction tools designed for time-dependent hyperbolic conservation laws require careful adaptation to other types of partial differential equations. In particular, the numerical treatment of source terms becomes important in the steady state limit, which may be affected by algebraic manipulations of the weighted residual formulation. The nonlinear stabilization term of the proposed method includes fluxes that modify the standard Galerkin discretization of source terms in an appropriate manner. The underlying design philosophy is based on an analogy with a well-balanced finite element scheme for the shallow water equations. It is hoped that this analogy (and the way in which it is exploited in the present paper) will advance the development of next-generation flux limiters for finite element discretizations of balance laws.

\medskip
\paragraph{\bf Acknowledgments}
The work of D. Kuzmin was supported by the German Research Foundation (Deutsche
Forschungsgemeinschaft, DFG) under grant KU 1530/23-3. {The work
of P. Knobloch was supported by the grant No.~22-01591S of the Czech Science
Foundation.}

 \bibliographystyle{plain}
\bibliography{MCL}      

\end{document}

%% file: preamble.tex
\usepackage{chemformula} 
\usepackage[T1]{fontenc} 
\usepackage{amssymb}
\usepackage{amsthm}
\usepackage{amsmath}
\usepackage{mathtools}
\usepackage{subcaption}
\usepackage{color}
\usepackage{tikz}
\usepackage{epic}
\usepackage{pgfplots}
\usetikzlibrary{calc,patterns}
\usetikzlibrary{positioning} 
\usetikzlibrary{calc}
\usepackage{pgfplots}
\pgfplotsset{width=10cm,compat=1.9}
\usepackage{multirow}
\usepackage{rotating}
\usepackage{siunitx}
\pgfplotsset{every tick label/.append style={font=\large}}

\usepackage{hyperref}
\usepackage[all]{hypcap}
\usepackage[a4paper, total={6.5in, 8.5in}]{geometry}

\definecolor{light_gray}{gray}{0.75}
\definecolor{lighter_gray}{gray}{0.5}
\colorlet{light_blue}{blue!20}
\definecolor{dark_green}{rgb}{0.0, 0.6, 0.0}
\definecolor{royal_blue}{rgb}{0.0, 0.22, 0.66}
\definecolor{salmon}{rgb}{1.0, 0.55, 0.41}
\definecolor{gold}{rgb}{0.8, 0.63, 0.21}
\definecolor{navy_blue}{rgb}{0.0, 0.0, 0.5}
\definecolor{crimson}{rgb}{0.79, 0.0, 0.09}
\definecolor{amethyst}{rgb}{0.6, 0.4, 0.8}
\definecolor{alizarin}{rgb}{0.82, 0.1, 0.26}
\definecolor{amaranth}{rgb}{0.9, 0.17, 0.31}
\definecolor{azure}{rgb}{0.0, 0.5, 1.0}
\definecolor{canaryyellow}{rgb}{0.82, 0.41, 0.12}
\definecolor{carrotorange}{rgb}{0.8, 0.33, 0.0}
\definecolor{cadmiumgreen}{rgb}{0.0, 0.42, 0.24}
\definecolor{copper}{rgb}{0.72, 0.45, 0.2}
\definecolor{aqua}{rgb}{0.5, 1.0, 0.83}
\definecolor{awesome}{rgb}{1.0, 0.13, 0.32}
\definecolor{candyapplered}{rgb}{1.0, 0.03, 0.0}
\definecolor{caribbeangreen}{rgb}{0.0, 0.8, 0.6}

%% file: macros.tex
\renewcommand{\vec}[1]{\ensuremath{\mathbf{#1}}}

\newcommand{\dx}{\mathrm{d}\mathbf{x}}

\newcommand{\bx}{\vec{x}}

\renewcommand{\iota}{\mathrm{i}}

\newcommand{\R}{\mathbb{R}}

\newcommand{\fictxj}{\bx^i_j}
\newcommand{\fictxk}{\bx^i_k}
\newcommand{\fictuj}{u^i_j}
\newcommand{\cellij}{K^i_j}

\newcommand{\beq}{\begin{equation}}
\newcommand{\eeq}{\end{equation}}

\newcommand{\Ni}{\mathcal N_i}

\newcommand{\Nis}{\mathcal N_i\backslash\{i\}}





\newcommand{\doi}[1]{\textrm{\textsc{doi:}} \href{http://dx.doi.org/#1}{\nolinkurl{#1}}}


%% file: numres.tex
In this section, we perform numerical studies for two-dimensional test problems. In our discussion of the results, the label MC is used for the monolithic convex limiter presented in Sec.~\ref{subsec:convective}. The label WMC refers to the well-balanced generalization of MC, as presented in Sec.~\ref{sec:well_balanced}. All simulations were performed using a \textsc{ParMooN} \cite{WB16} implementation of the methods under investigation. 

The square domain $\Omega=(0,1)^2$  is used in all of
 our numerical experiments. Uniform refinement of the coarse (level 0) triangulations shown in Fig.~\ref{fig:grids} yields two families of computational meshes. We use the label \textit{Grid~1} for meshes generated from the triangulation shown on the left and \textit{Grid~2} for refinements of the  triangulation shown on the right. The stopping criterion for fixed-point iterations uses the absolute tolerance $10^{-8}$ for the residual of the nonlinear discrete problem.

\begin{figure}[h!]\centering
\includegraphics[scale=0.15]{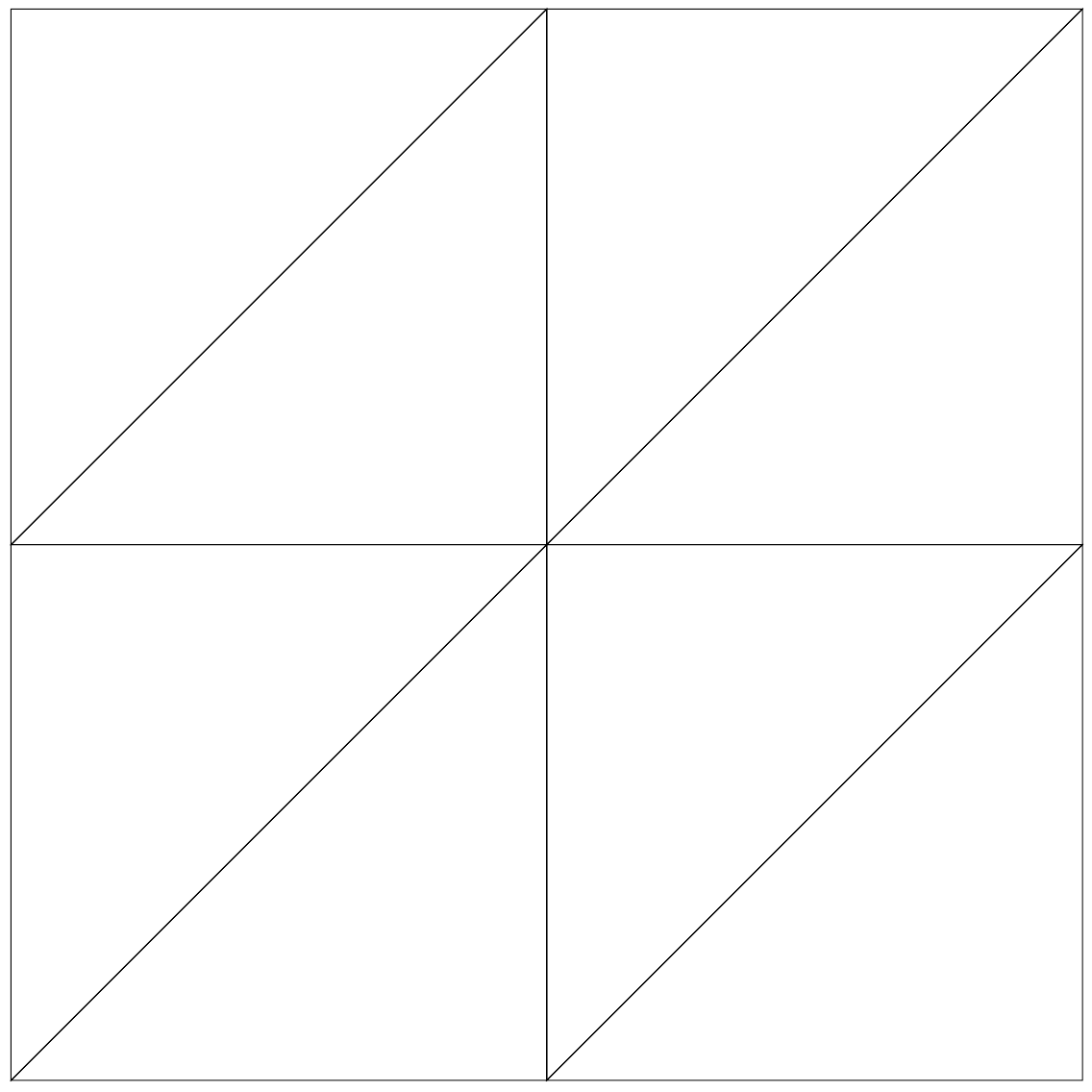}\hspace*{1em}
\includegraphics[scale=0.15]{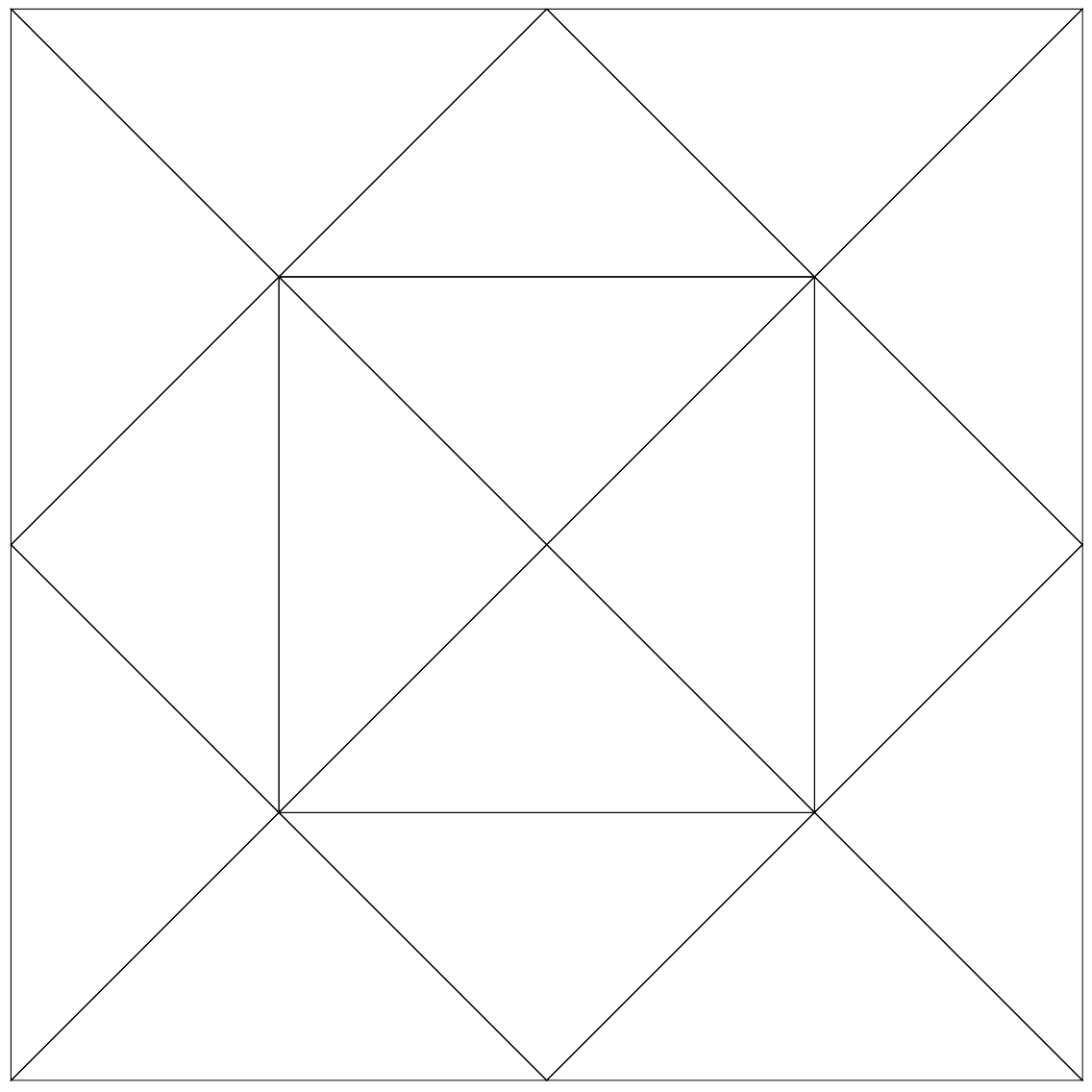}
\caption{Level 0 triangulations used for Grid 1 (left) and Grid 2 (right) families of computational meshes.}\label{fig:grids}
\end{figure}





\subsection{Interior layers}\label{ex:bal_reac}
In the first numerical example, we solve the CDR equation \eqref{eq:cdr_eqn} with $\mathbf{v}=(1,0)^\top$ and $\varepsilon=10^{-8}$. 
To demonstrate the need for a well-balanced treatment of source terms, we set
$$
f(x,y)=\begin{cases}
10 & \mbox{if}\ x\in[0.1,0.6],\ y\in[0.25,0.75],\\
 0 & \mbox{otherwise},
\end{cases}\qquad
 c(x,y)=\begin{cases}
25 & \mbox{if}\ x > 0.75,\\
0 & \mbox{otherwise}.
\end{cases}
 $$
Homogeneous Dirichlet boundary conditions are prescribed on $\Gamma_D=\Gamma$. The discontinuities in $f$ and $c$ produce sharp interior
layers. The exact solution of this new test problem is linear in the core of the subdomain $(0.1,0.6)\times (0.25,0.75)$ and constant
in the core of the subdomain $(0.6,0.75)\times (0.25,0.75)$. Note that the restriction of \eqref{eq:cdr_eqn} to the former subdomain is a CDR equation of
the form considered at the beginning of Sec.~\ref{sec:well_balanced}. 

\begin{figure}[h!]\centering
\includegraphics[width=0.3\textwidth]{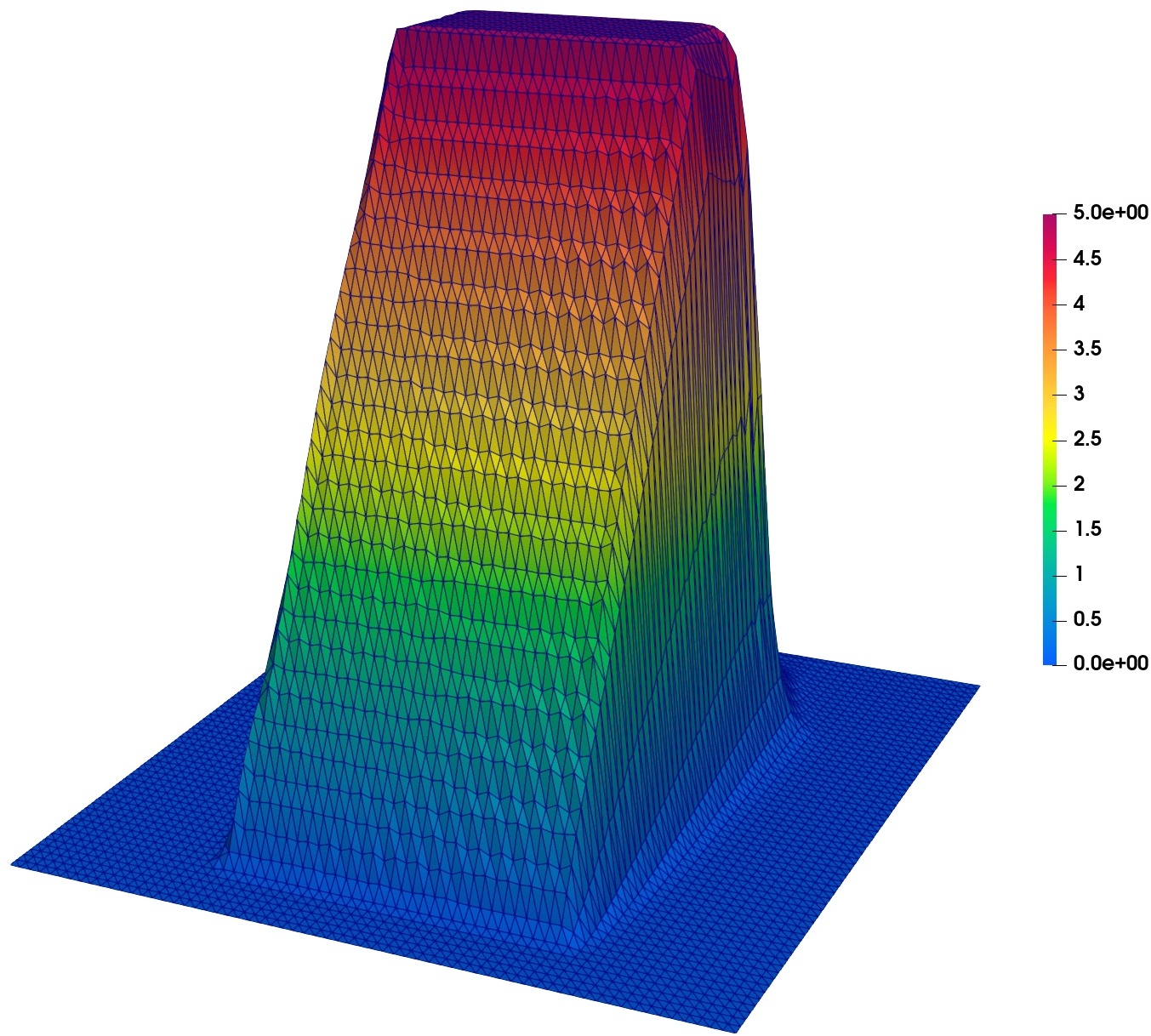}\hspace*{1em}
\includegraphics[width=0.3\textwidth]{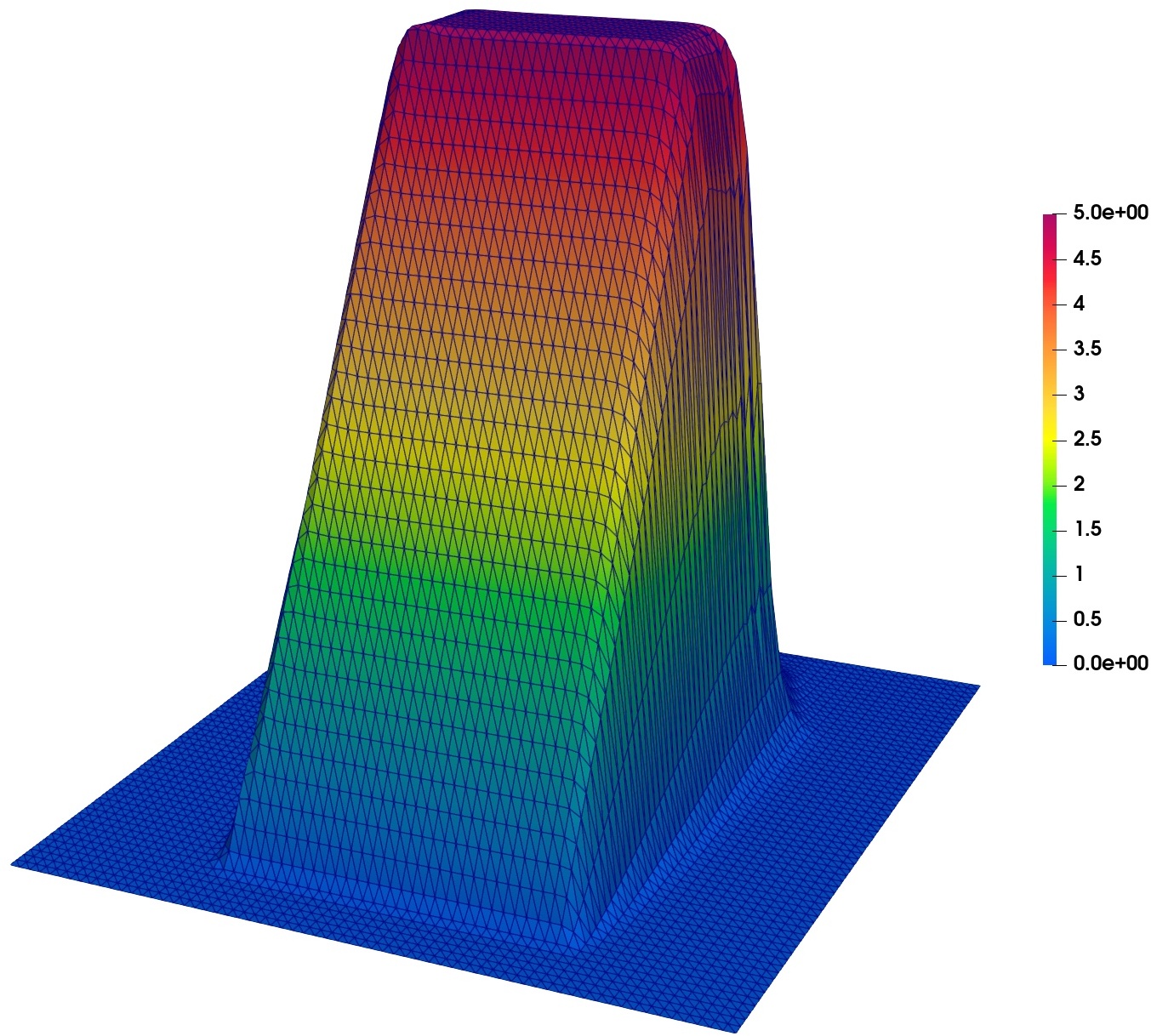}
\caption{Interior layers,  MC (left) and WMC (right) solutions, Grid 1 / level~5.}\label{fig:balanced_reaction_wmc_g1}
\vskip0.5cm

\includegraphics[width=0.3\textwidth]{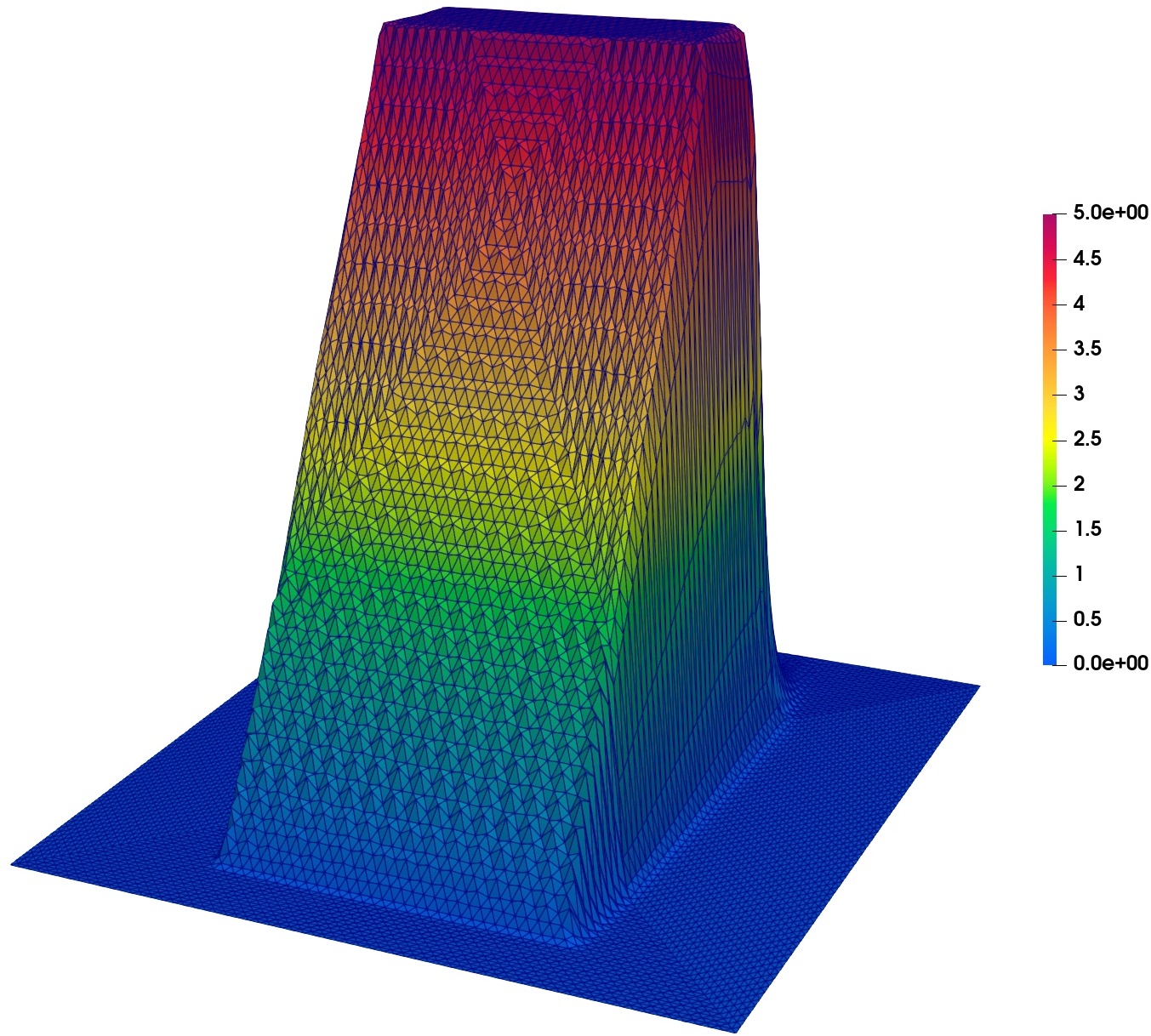}\hspace*{1em}
\includegraphics[width=0.3\textwidth]{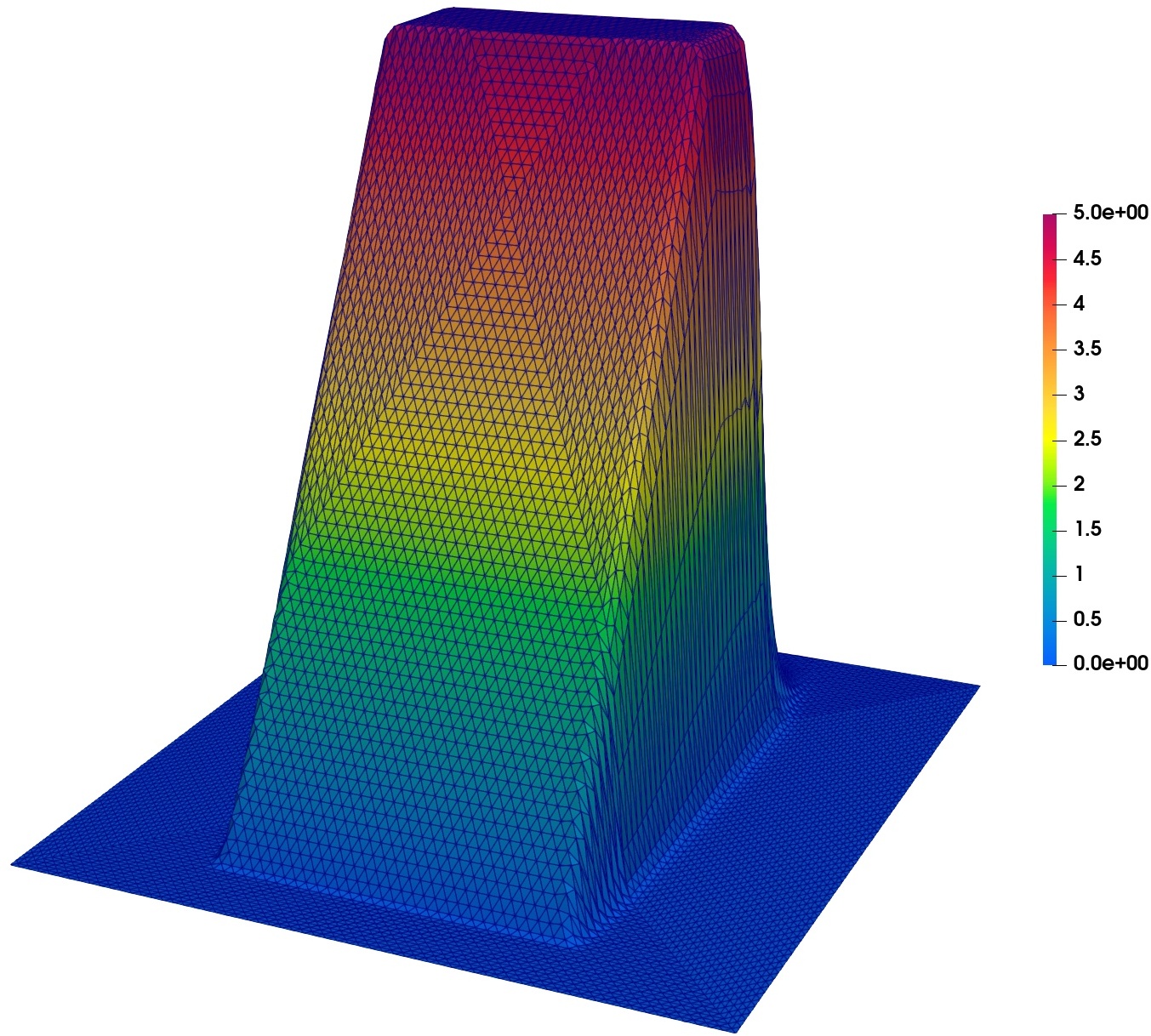}
\caption{Interior layers,  MC (left) and WMC (right) solutions, Grid 2 / level~5.}\label{fig:balanced_reaction_wmc_g2}
\end{figure}

The numerical solutions shown in Figs \ref{fig:balanced_reaction_wmc_g1} and \ref{fig:balanced_reaction_wmc_g2} were obtained on level 5 triangulations
of Grid~1 and Grid 2, respectively. The spurious ripples in the MC results are caused by the fact that the flux-corrected approximation to the
convective term is not in equilibrium with the standard Galerkin discretization of the source term. The WMC version is free of this drawback and
produces nonoscillatory results. Figure \ref{fig:balance_reaction_plot} shows the Grid 2 / level 7 approximation obtained with WMC.

\begin{figure}[h!]
\centering
\includegraphics[width=0.3\textwidth]{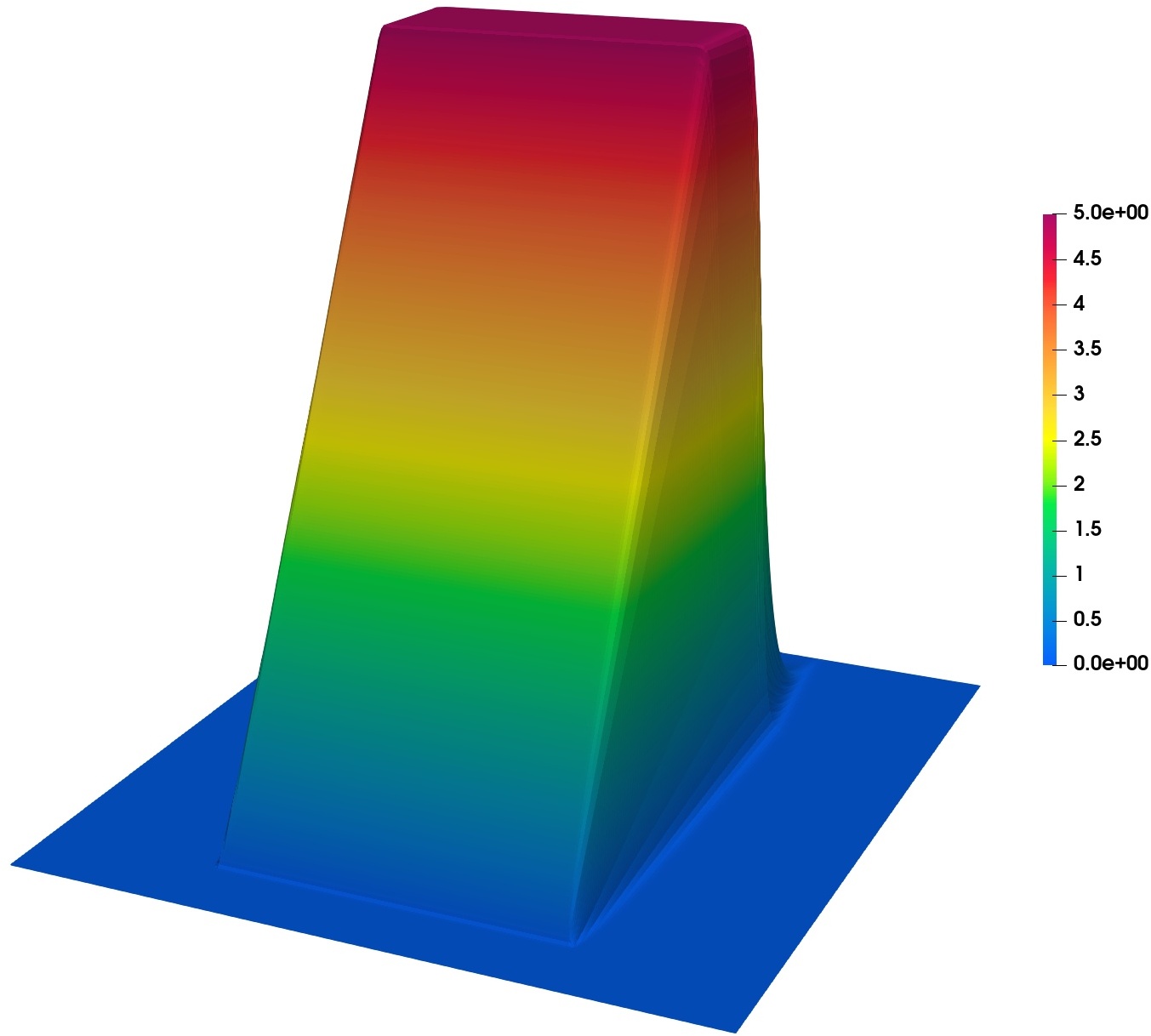}
\caption{Interior layers, WMC solution, Grid 2 / level~7.}\label{fig:balance_reaction_plot}
\end{figure}

%

\subsection{Boundary layers}\label{ex:boundary_layer}
The next test problem that we consider in this numerical study was introduced by John et al. in \cite[Example~3]{JMT96}.
The manufactured
exact solution
$$
u(x,y)=xy^2-y^2\exp\left(\frac{2(x-1)}{\varepsilon}\right)-x\exp\left(\frac{3(y-1)}{\varepsilon}\right)+\exp\left(\frac{2(x-1)+3(y-1)}{\varepsilon}\right)
$$
of the CDR equation \eqref{eq:cdr_eqn} with $\mathbf{v}=(2,3)^\top$ and $c\equiv 0$ is used to define the right-hand side $f$ and
the Dirichlet boundary data $u_D$. The exact solution has boundary layers at $x=1$ and $y=1$.

\begin{figure}[h!]\centering

\includegraphics[width=0.3\textwidth]{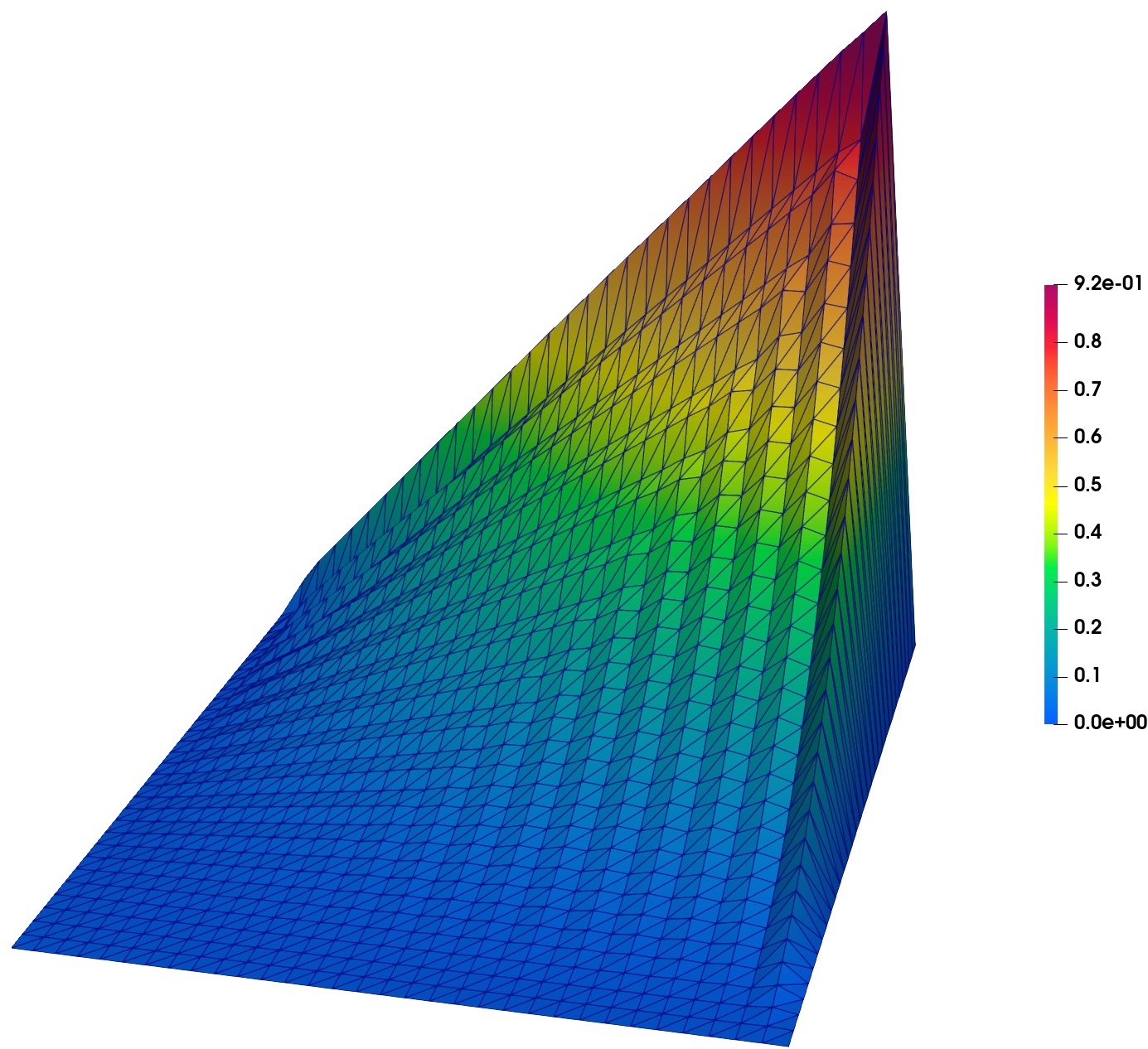}\hspace*{1em}
\includegraphics[width=0.3\textwidth]{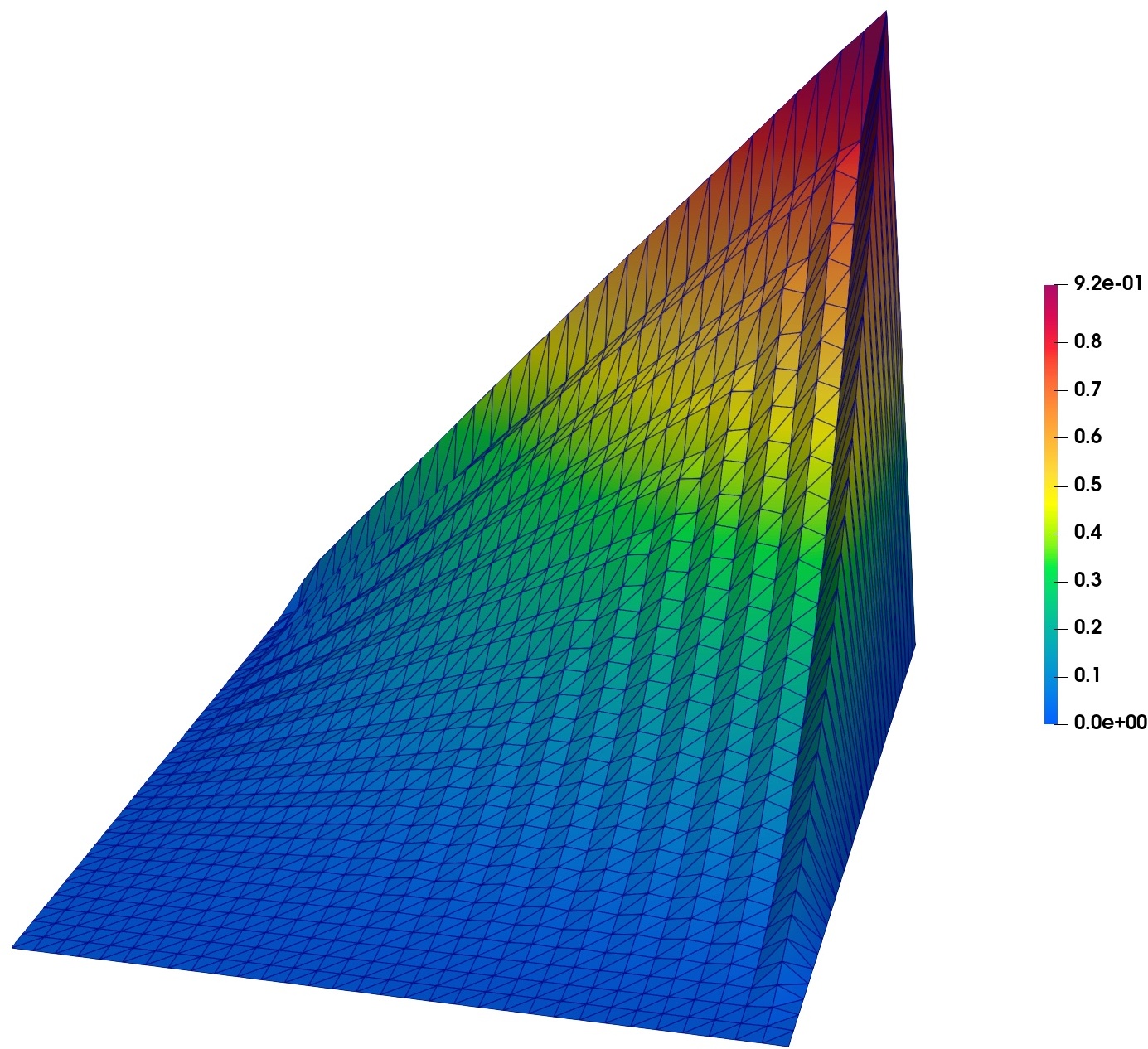}\hspace*{1em}
\includegraphics[width=0.3\textwidth]{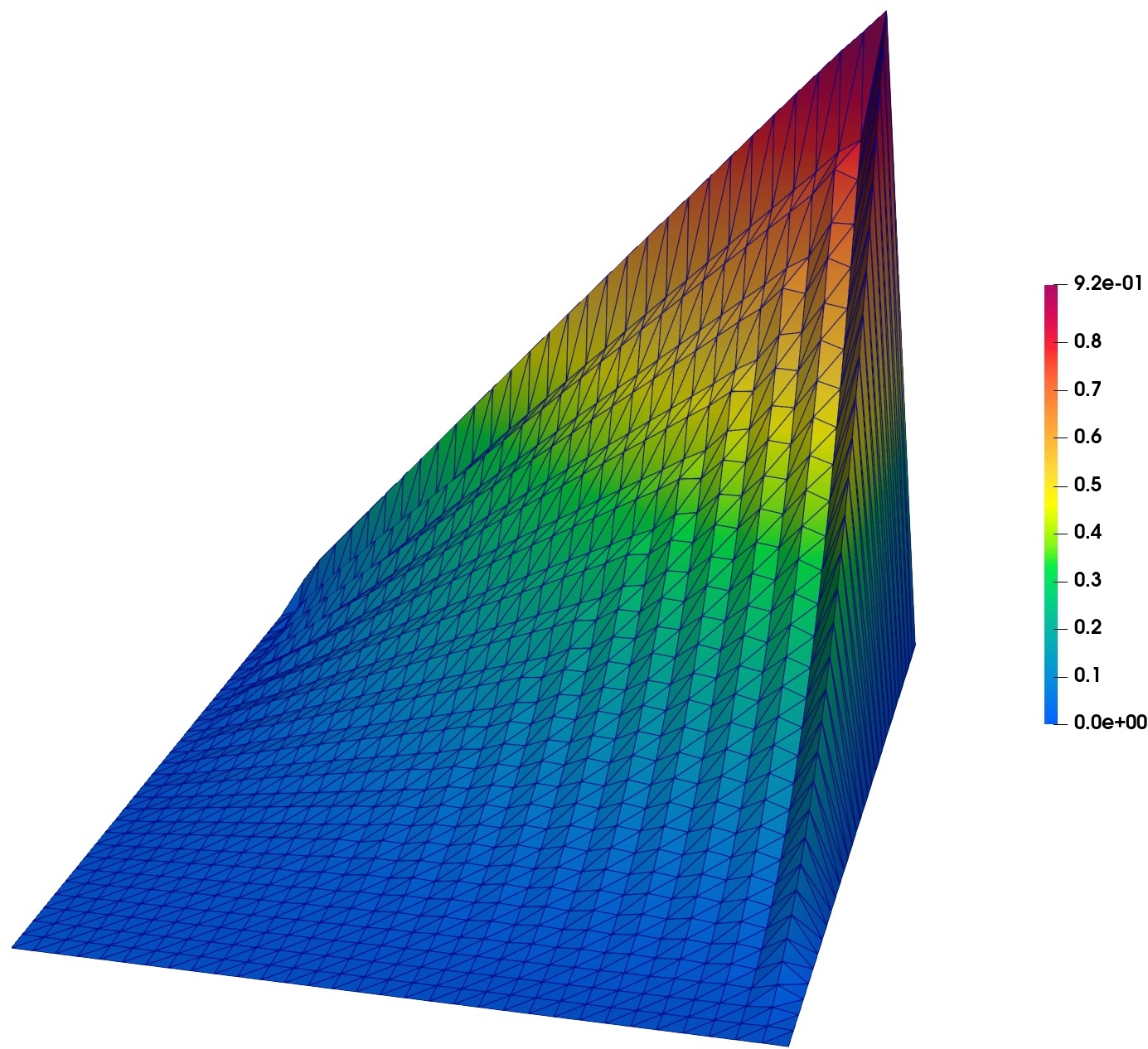}
\caption{Boundary layers, MC solutions,  Grid 1 / level~4, $\varepsilon=10^{-3}, 10^{-6},$ $10^{-9}$ (left to right).}\label{fig:boundary_layer_mc_diff_eps}
\vskip0.5cm

\includegraphics[width=0.3\textwidth]{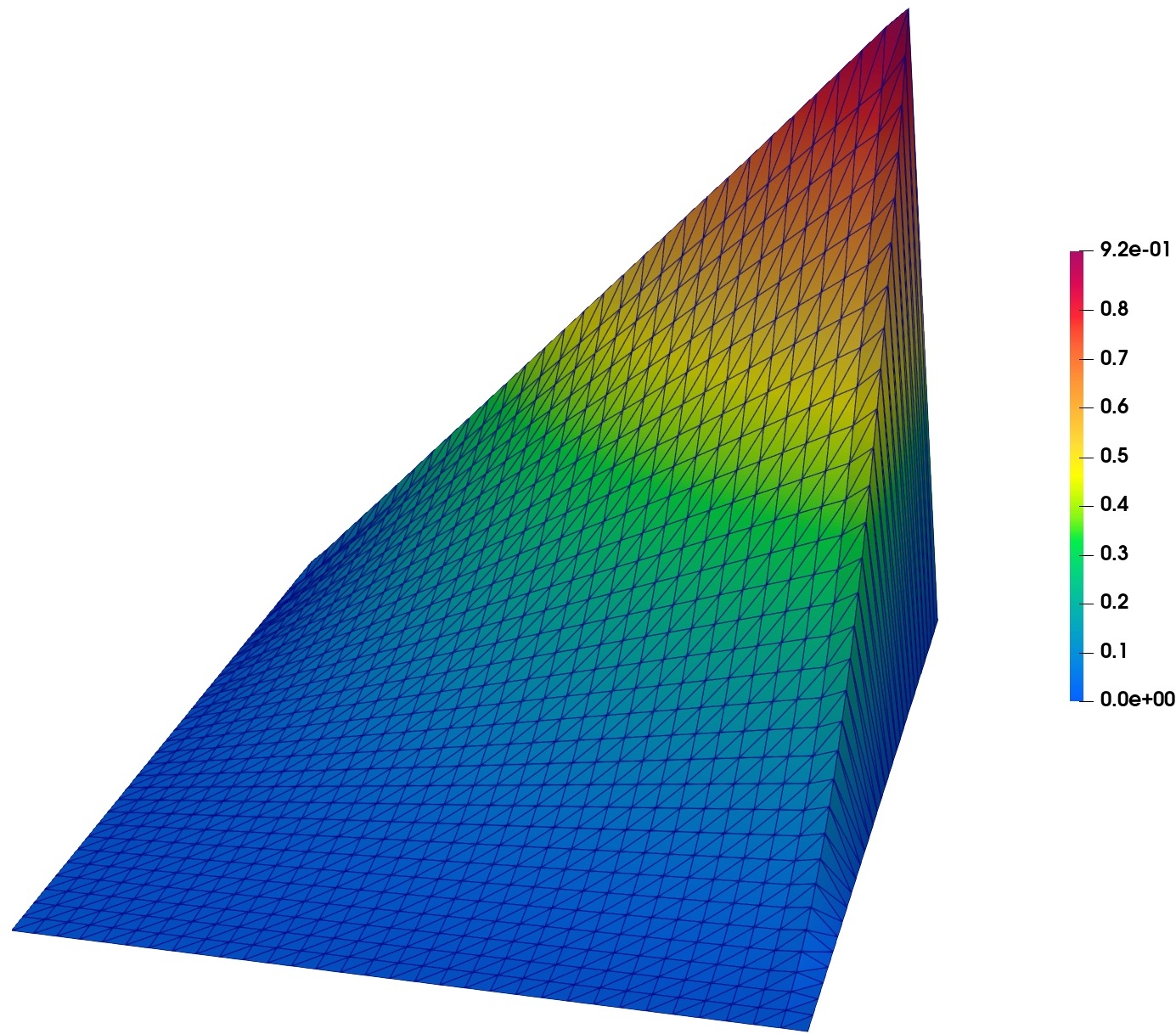}\hspace*{1em}
\includegraphics[width=0.3\textwidth]{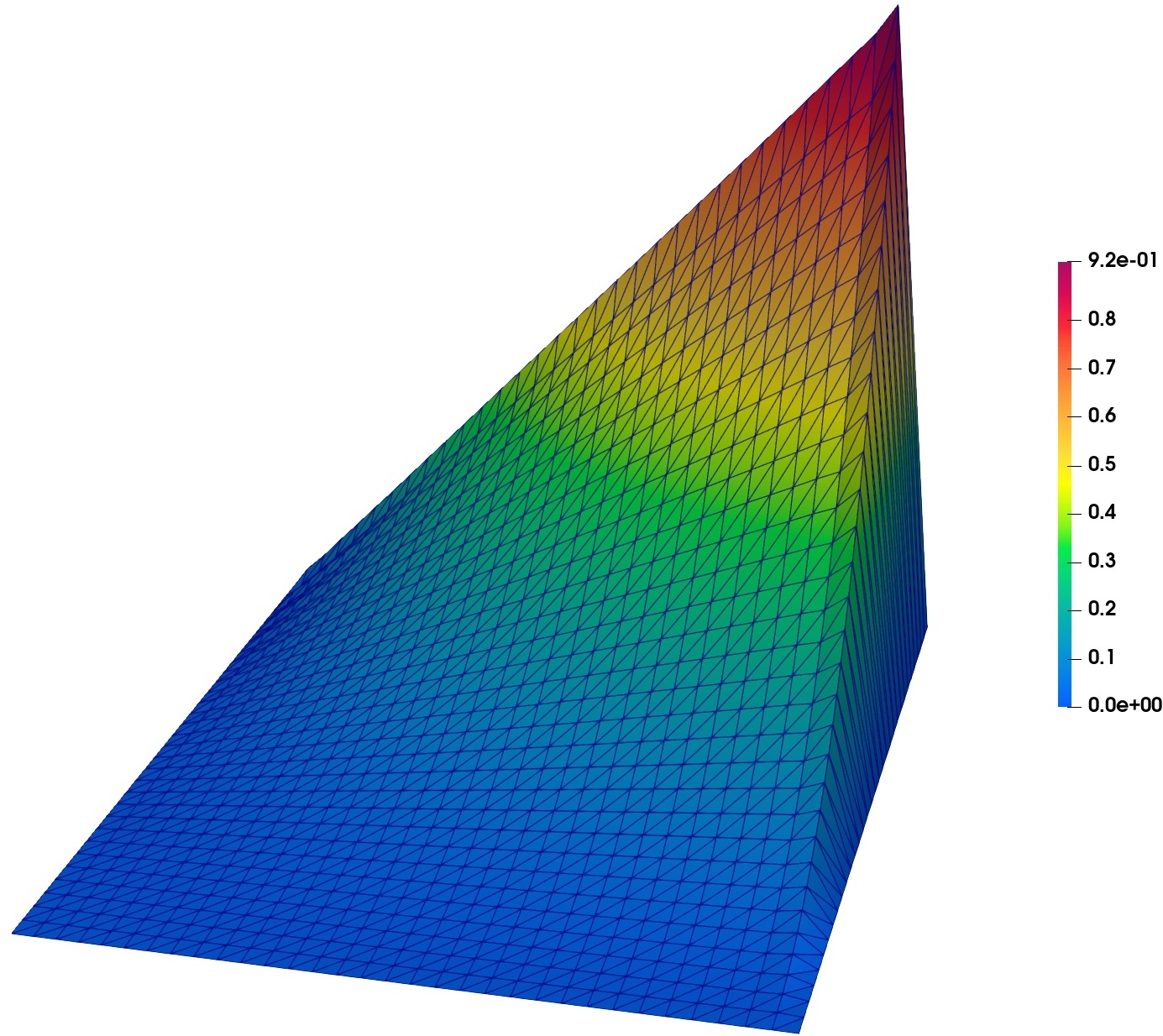}\hspace*{1em}
\includegraphics[width=0.3\textwidth]{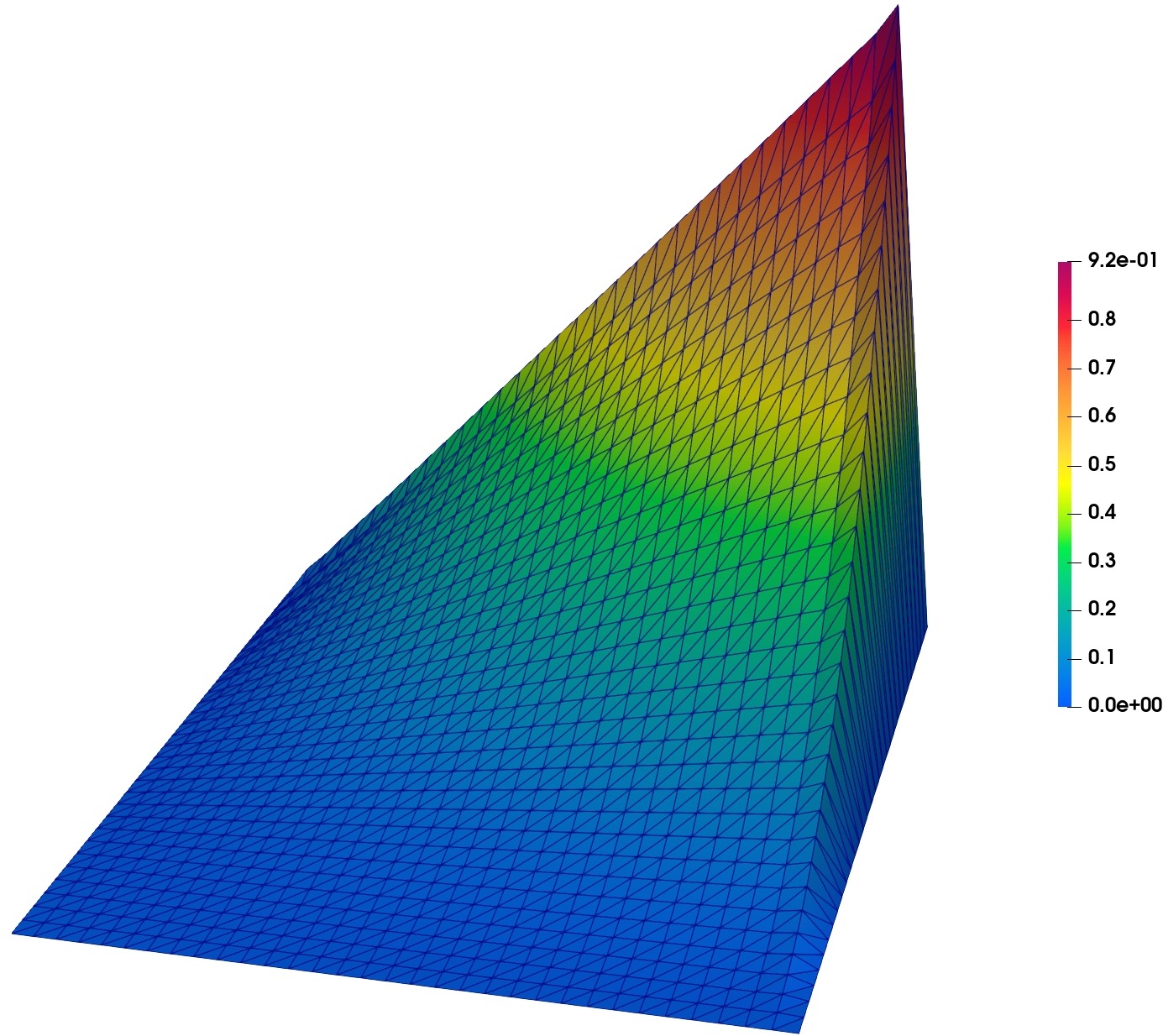}
\caption{Boundary layers, WMC solutions,  Grid 1 / level~4, $\varepsilon=10^{-3}, 10^{-6},$ $10^{-9}$ (left to right).}
\label{fig:boundary_layer_wmc_diff_eps}
\end{figure}

We ran numerical simulations for $\varepsilon\in \left\lbrace 10^{-3}, 10^{-6},10^{-9}\right\rbrace$ on Grid 1 / level 4. The MC and WMC results are presented in Figs~\ref{fig:boundary_layer_mc_diff_eps} and \ref{fig:boundary_layer_wmc_diff_eps}, respectively. Once again, the MC version produces spurious oscillations, whereas the WMC approximations are well resolved and free of ripples. Figure \ref{fig:boundary_layer} shows the WMC result for $\varepsilon=10^{-9}$ obtained using Grid 1 / level~7.

\begin{figure}[h!]\centering
\includegraphics[width=0.3\textwidth]{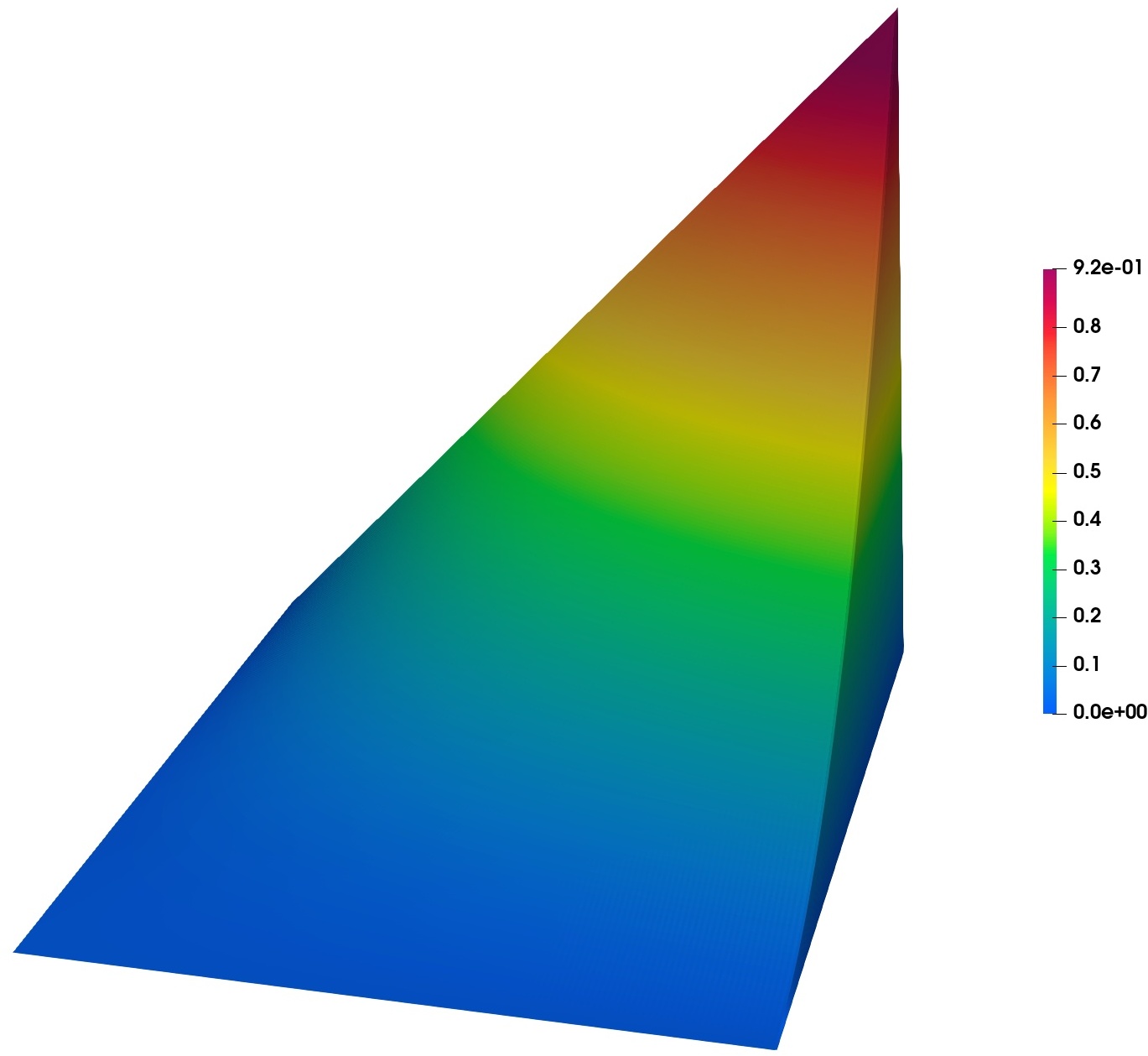}
\caption{Boundary layers,  WMC solution,  Grid 1 / level~7, $\varepsilon=10^{-9}$.}\label{fig:boundary_layer}
\end{figure}

%

\subsection{Circular layers}\label{ex:rotating_body}
In the next test, we perform numerical experiments for
the CDR equation \eqref{eq:cdr_eqn} with 
the non-constant velocity field $\mathbf{v}=(y,-x)^\top$. The source terms
are defined by
$$
f(x,y)=\left\lbrace\begin{array}{ll}
1 &\mathrm{if}\ \ 0.25\leq \sqrt{x^2+y^2}\leq 0.75,\\
0 &\mathrm{otherwise},
\end{array}\right.
$$
and $c(x,y)=1-f(x,y)$. 
We impose a homogeneous Neumann boundary condition on
$\Gamma_N=\{(x,0)\,:\, 0 < x < 1\}$ and set $u_D=0$ on the
Dirichlet boundary $\Gamma_D=\Gamma\backslash\Gamma_N$.
 
\begin{figure}[h!]\centering

\includegraphics[width=0.3\textwidth]{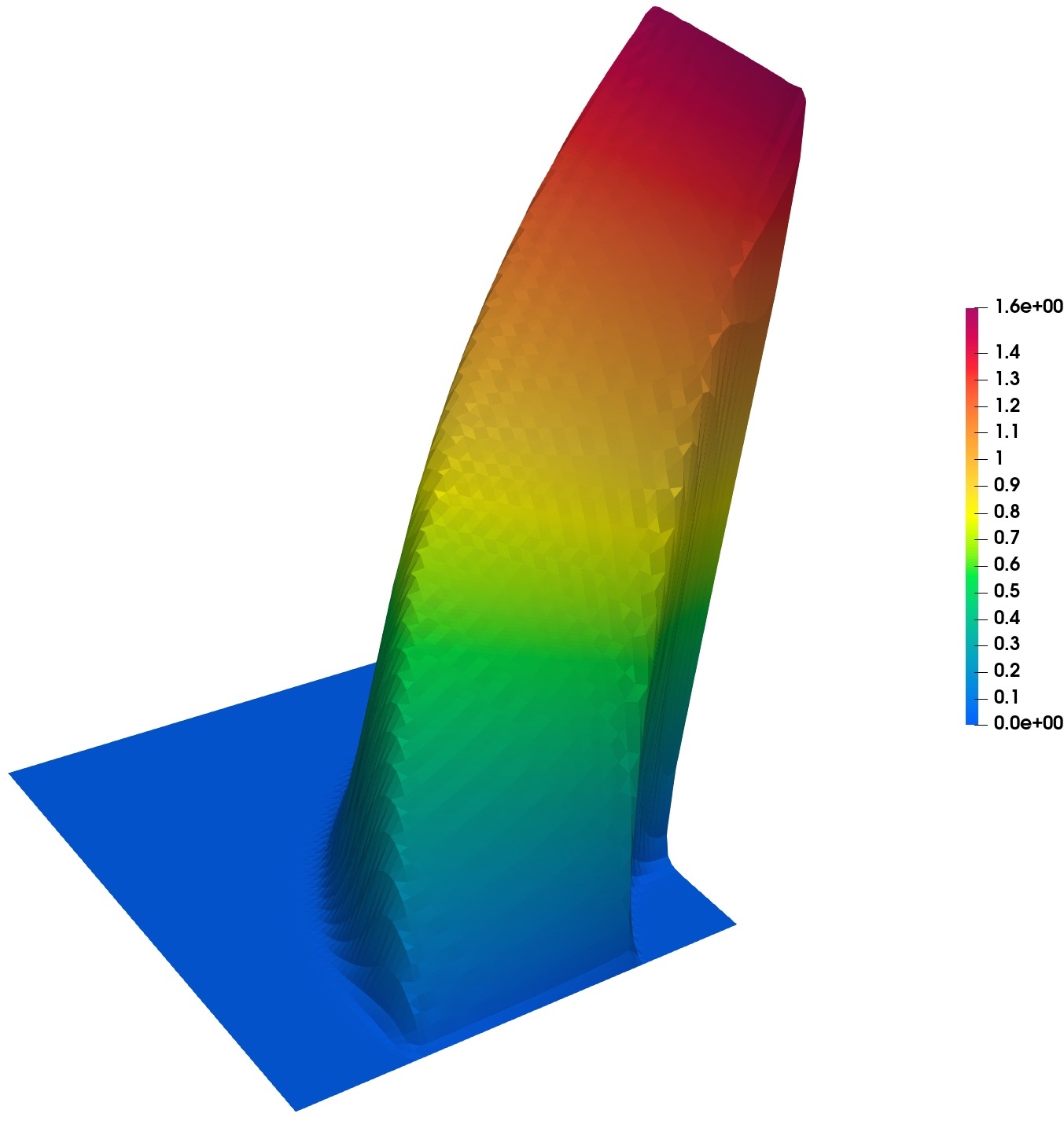}\hspace*{1em}
\includegraphics[width=0.3\textwidth]{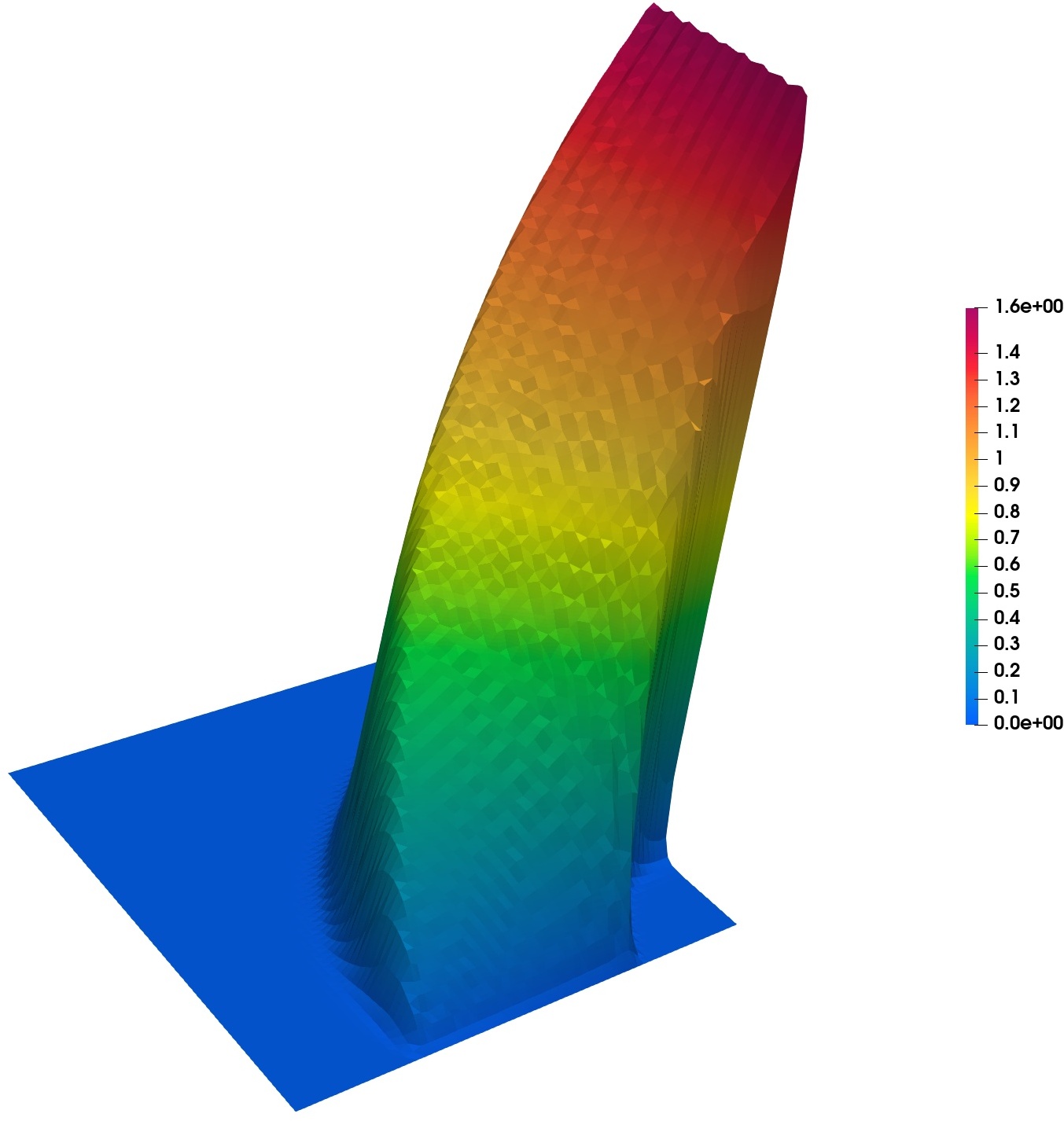}
\caption{Circular layers, MC solutions,  Grid~1 / level~5, $\varepsilon=10^{-4}, 10^{-6}$ (left to right).}\label{fig:rotating_body_mc_diff_eps}
\vskip0.5cm

\includegraphics[width=0.3\textwidth]{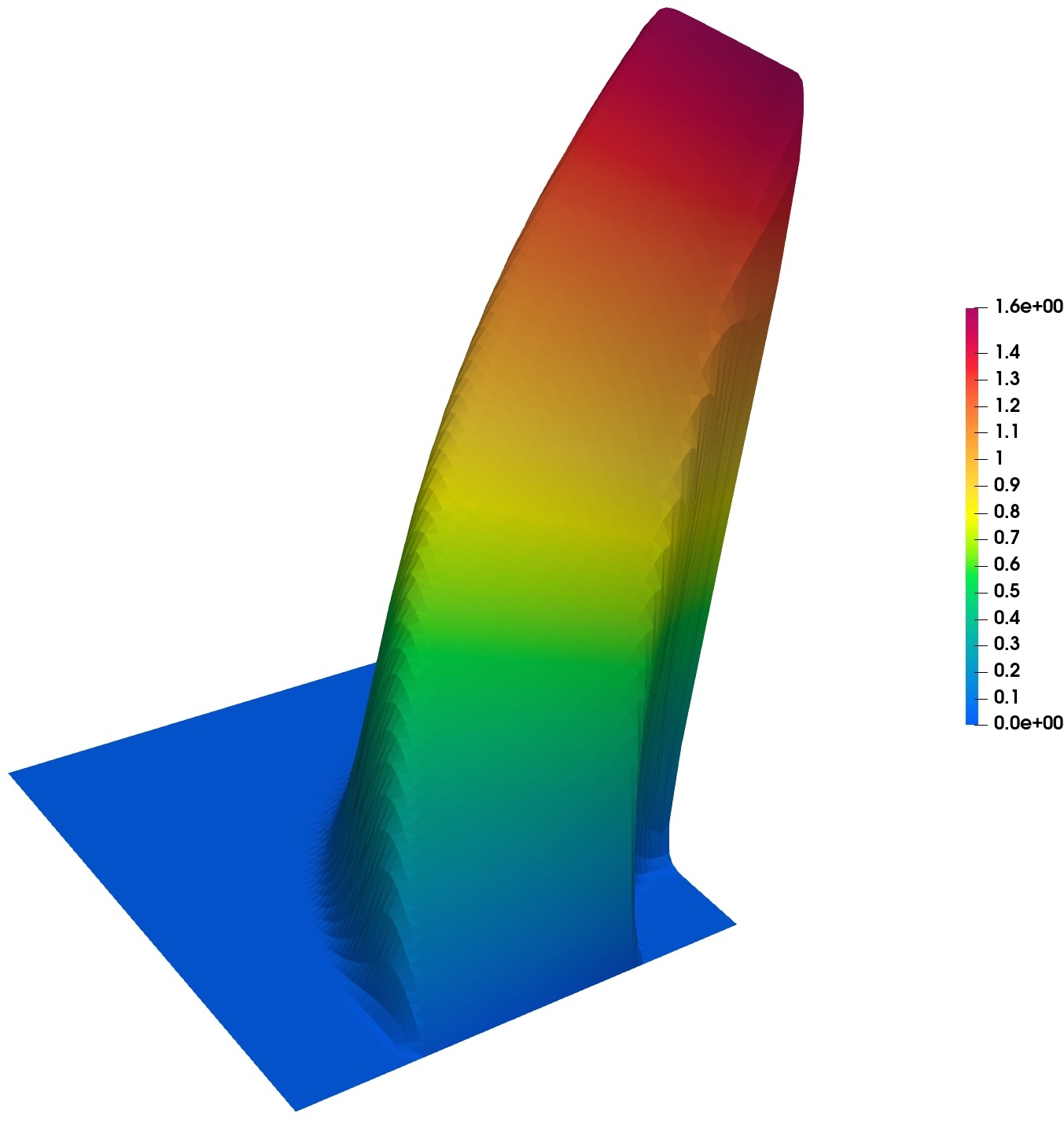}\hspace*{1em}
\includegraphics[width=0.3\textwidth]{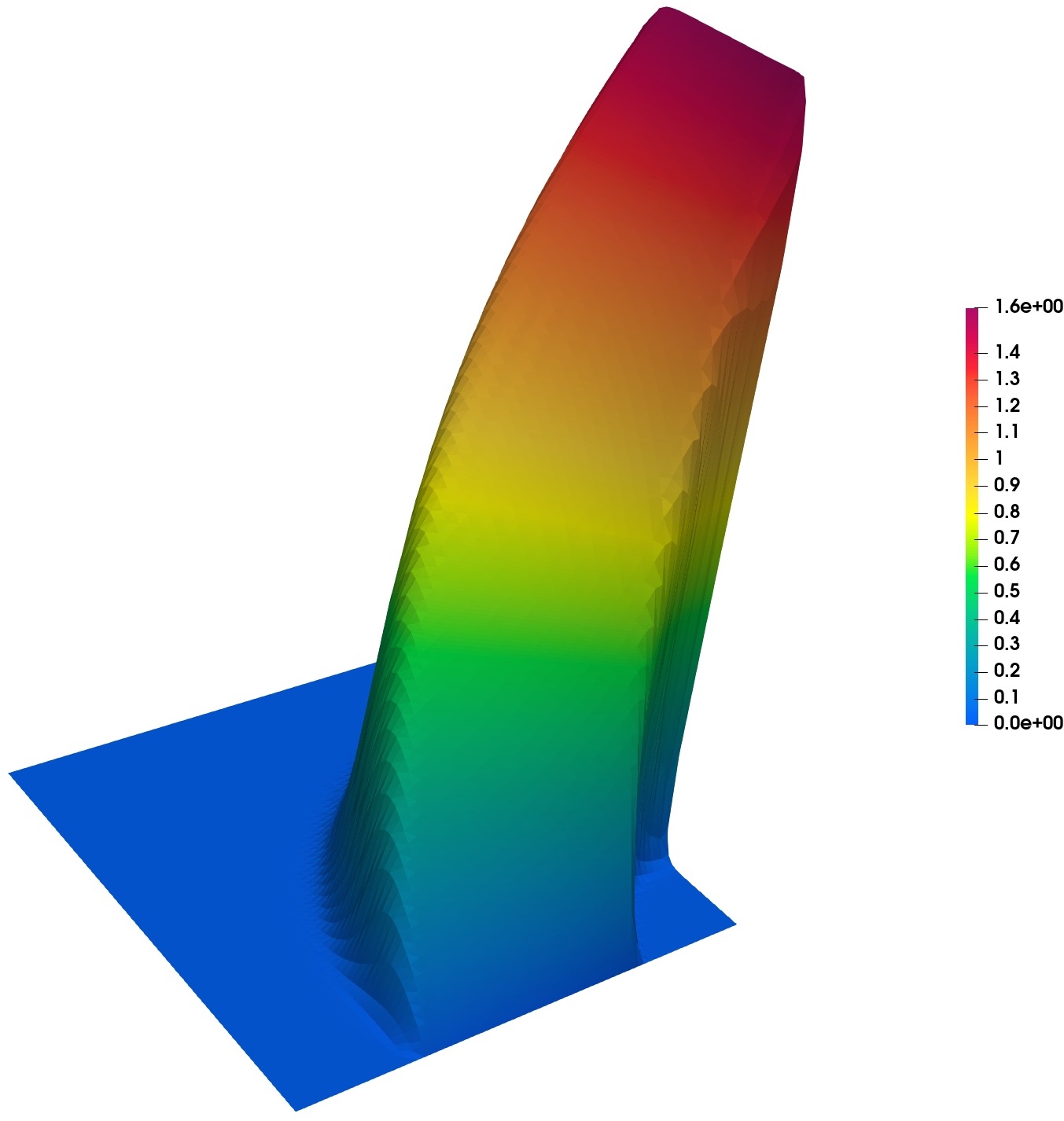}
\caption{Circular layers, WMC solutions,  Grid~1 / level~5, $\varepsilon=10^{-4}, 10^{-6}$ (left to right).}\label{fig:rotating_body_wmc_diff_eps}
\vskip0.5cm
\end{figure}

We ran numerical simulations for $\varepsilon\in \left\lbrace 10^{-4}, 10^{-6}\right\rbrace$ on Grid 1 / level 5. The MC and WMC results
are presented in Figs~\ref{fig:rotating_body_mc_diff_eps} and \ref{fig:rotating_body_wmc_diff_eps}, respectively. Spurious ripples can again be seen in the MC solution of the CDR equation with
$\varepsilon=10^{-6}$. Although the exact solution is not linear between the circular internal layers, the WMC approximation is free of ripples. Figure \ref{fig:rotating_body} shows the WMC result for $\varepsilon=10^{-4}$ obtained using Grid 1 / level~7.

\begin{figure}[h!]\centering
\includegraphics[width=0.3\textwidth]{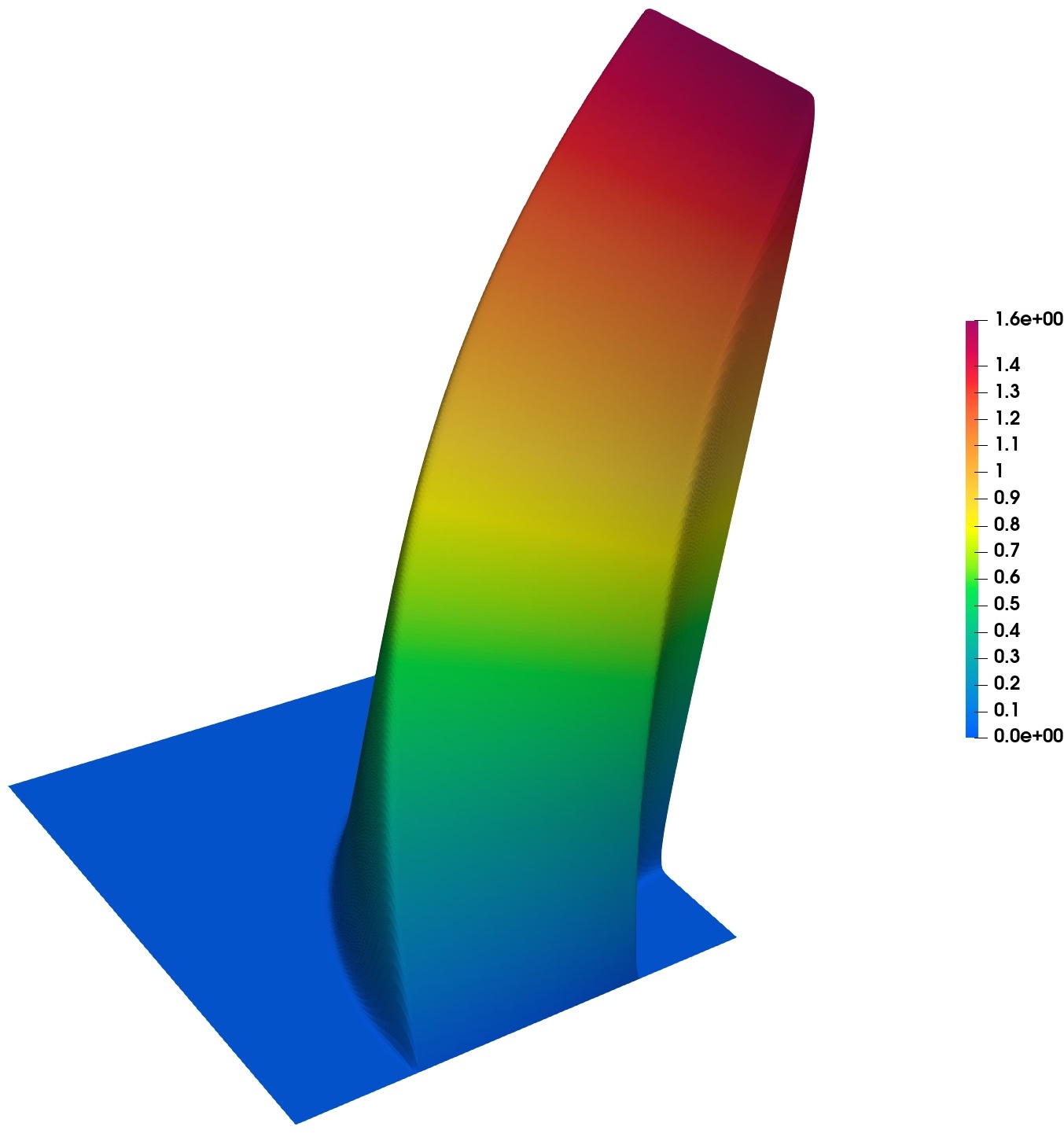}
\caption{Circular layers,  WMC solution,  Grid 1 / level~7, $\varepsilon=10^{-4}$.}\label{fig:rotating_body}
\end{figure}

\subsection{Circular convection}\label{ex:circular_convection}
In this final example, we study the grid convergence properties of the WMC method. Following Lohmann \cite[Eq.~(3.24)]{Loh19}, we
consider equation~\eqref{eq:cdr_eqn} with $\varepsilon=0, \mathbf{v}=(y,-x)^\top$, and $c\equiv 1$. The smooth exact solution and the inflow boundary condition are given by
$$
u(x,y)=\exp\left(-100\left(\sqrt{x^2+y^2}-0.7\right)^2\right),\qquad 0\leq x,y\leq 1.
$$
The right-hand side satisfies $f=cu$.

Figure \ref{fig:circular_convection} shows the Grid 1 / level 7 solution obtained using the WMC limiter.
The experimental order of convergence (E.O.C.) w.r.t. a norm $\|\cdot\|$
is determined using the formula
$$
\mathrm{E.O.C.}=\log_2\left( \frac{\|u-u_{2h}\|}{\|u-u_{h}\|}\right).
$$

\begin{figure}[h!]\centering
\includegraphics[width=0.375\textwidth]{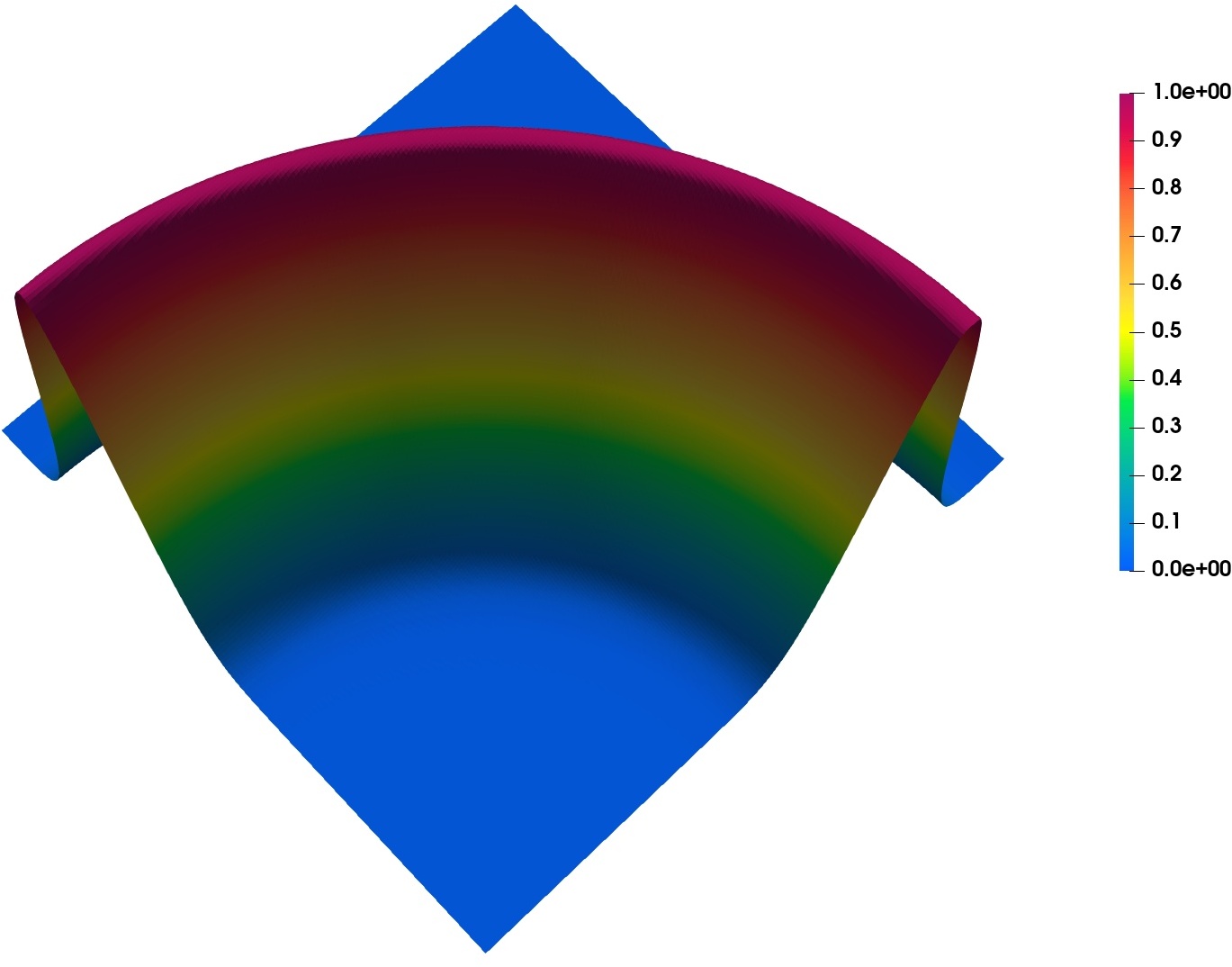}
\caption{Circular convection, WMC solution, Grid 1 / level~7.}\label{fig:circular_convection}
\end{figure}

In Tables \ref{tab:circular_convection_error_grid1} and \ref{tab:circular_convection_error_grid2}, we list the $L^1$ and $L^2$ errors for both types of computational meshes. The Grid 1 and Grid 2 convergence rates approach 2 on fine mesh levels. We conclude that the WMC treatment of source terms does not degrade the convergence behavior of our scheme. Further improvements could be achieved using linearity-preserving local bounds (cf.~\cite[Example 6.1]{Ku20}).

\begin{table}[h!]
    \centering
    \begin{tabular}{|ccccc|}
       \hline
        \textrm{Level} & $\|u-u_h\|_{L^2}$ & E.O.C. & $\|u-u_h\|_{L^1}$ & E.O.C. \\ 
        \hline \hline
3& \num{0.095276909014531}&\num{0}&\num{0.057723388291951}&\num{0}\\
4& \num{0.038372306081693}&\num{1.3120611398796116}&\num{0.018893367506366}&\num{1.6112761197550955}\\
5& \num{0.01451191242234}&\num{1.402827815474684}&\num{0.0057690092297079}&\num{1.71148439033664}\\
6& \num{0.0045088527953071}&\num{1.6864053396570695}&\num{0.0015053242098999}&\num{1.9382493299718966}\\
7& \num{0.0012726041914149}&\num{1.8249766314115696}&\num{0.00032900036577443}&\num{2.193913148537673}\\
8& \num{0.00029652705261797}&\num{2.101548143161334}&\num{0.000059768047877223}&\num{2.460642860061867}\\
\hline
    \end{tabular}
    \caption{Circular convection, $\|\cdot\|_{L^2}$ and $\|\cdot\|_{L^1}$ errors for Grid~1 triangulations.}
    \label{tab:circular_convection_error_grid1}
\end{table}

\begin{table}[h!]
    \centering
    \begin{tabular}{|ccccc|}
       \hline
        \textrm{Level} & $\|u-u_h\|_{L^2}$ & E.O.C. & $\|u-u_h\|_{L^1}$ & E.O.C. \\ 
        \hline \hline
3 & \num{ 0.059103898044568 }&\num{ 0 }& \num{ 0.032231134811308 }& \num{ 0 }\\
4 & \num{ 0.023430355995447 }&\num{ 1.334874408412071 }& \num{ 0.010010115253685 }& \num{ 1.686996399423337 }\\
5 & \num{ 0.0080937855150686 }&\num{ 1.5334923510681087 }& \num{ 0.0029157734774828 }& \num{ 1.779508036964343 }\\
6 & \num{ 0.0024763258847636 }&\num{ 1.7086134322890378 }& \num{ 0.00075833701014984 }& \num{ 1.9429676034311334 }\\
7 & \num{ 0.00068036547657055 }&\num{ 1.8638193438233794 }& \num{ 0.00015363394740536 }& \num{ 2.303342101048385 }\\
8 & \num{ 0.00016789195541868 }&\num{ 5.3407049269569775 }& \num{ 0.000007614616136816 }& \num{ 4.334581912843703 }\\
\hline
    \end{tabular}
    \caption{Circular convection, $\|\cdot\|_{L^2}$ and $\|\cdot\|_{L^1}$ errors for Grid~2 triangulations.}
    \label{tab:circular_convection_error_grid2}
\end{table}